\documentclass[12pt, a4paper]{amsart}

\usepackage[english]{babel}
\usepackage[T1]{fontenc}
\usepackage{lmodern}
\usepackage[utf8]{inputenc}
\usepackage{microtype}
\usepackage{amsmath}
\usepackage{amssymb}
\usepackage{amsthm}
\usepackage{enumerate}
\usepackage{caption}
\usepackage{multirow}
\usepackage{tikz}
\usetikzlibrary{cd}
\usetikzlibrary{arrows,decorations.markings}
\usetikzlibrary{matrix}
\usetikzlibrary{graphs}
\usetikzlibrary{backgrounds}
\usetikzlibrary{quotes}
\usepackage{mathrsfs}
\usepackage[autostyle=true]{csquotes}
\usepackage{leftidx}
\usepackage{varwidth}
\usepackage{hyperref}
\usepackage{cleveref}

\renewcommand{\arraystretch}{1.1}
\calclayout

\newcommand*\Dynkindots{\hbox to 2em{.\hss.\hss.}}
\def\DynkinNodeSize{1.5mm}
\def\DynkinDoubleArrowLength{3.25mm}
\def\DynkinTripleArrowLength{3.5mm}
\tikzset{
  bigdnode/.style={
    circle,
    inner sep=0pt,
    minimum size=2*\DynkinNodeSize,
    fill=white,
    draw},
  dnode/.style={
    circle,
    inner sep=0pt,
    minimum size=\DynkinNodeSize,
    fill=white,
    draw},
  bnode/.style={
    circle,
    inner sep=0pt,
    minimum size=\DynkinNodeSize,
    fill=black,
    draw},
  middlearrow/.style={
    decoration={markings,
      mark=at position 0.7 with
      {\draw (0:0mm) -- +(+160:\DynkinDoubleArrowLength); \draw (0:0mm) -- +(-160:\DynkinDoubleArrowLength);},
    },
    postaction={decorate}
  },
  triplemiddlearrow/.style={
    decoration={markings,
      mark=at position 0.7 with
      {\draw (0:0mm) -- +(+160:\DynkinTripleArrowLength); \draw (0:0mm) -- +(-160:\DynkinTripleArrowLength);},
    },
    postaction={decorate}
  },	
  sedge/.style={
  },
  dedge/.style={
    middlearrow,
    double distance=1mm,
  },
  tedge/.style={
    triplemiddlearrow,
    double distance=1.0mm+\pgflinewidth,
    postaction={draw},
  },
}
\makeatletter
\newcommand*\xleftrightarrow[2][]{%
  \ext@arrow 9999{\longleftrightarrowfill@}{#1}{#2}}
\newcommand*\longleftrightarrowfill@{%
  \arrowfill@\leftarrow\relbar\rightarrow}
\makeatother
\definecolor{asmpgray}{gray}{0.85}
\newsavebox{\asmpbox}
\newenvironment{assumption}{
\vspace{5pt}
\begin{lrbox}{\asmpbox}
\begin{varwidth}{13cm}
\ignorespaces
}
{
\end{varwidth}
\end{lrbox}
\begin{center}
\setlength{\fboxsep}{5pt}
\fcolorbox{black}{asmpgray}{\usebox{\asmpbox}}
\end{center}
\ignorespacesafterend
\vspace{5pt}
}

\newtheorem{Thm}{Theorem}[section]

\newtheorem{Lm}[Thm]{Lemma}
\newtheorem{Cor}[Thm]{Corollary}
\theoremstyle{definition}
\newtheorem{0}[Thm]{}

\newtheorem{Rem}[Thm]{Remark}

\newtheorem*{Not}{Notation}
\newtheorem*{Ack}{Acknowledgements}
\numberwithin{equation}{Thm}
\newcommand*{\mf}{\mathbf}
\newcommand*{\rom}[1]{\uppercase\expandafter{\romannumeral#1}}
\newcommand*{\C}{\mathbb{C}}
\newcommand*{\F}{\mathbb{F}}
\newcommand*{\Fpbar}{\overline{\F}_{p}}
\newcommand*{\Fp}{\F_{p}}
\newcommand*{\Fq}{\F_{q}}

\newcommand*{\Q}{\mathbb{Q}}
\newcommand*{\Ql}{\Q_\ell}
\newcommand*{\Qlbar}{\overline{\Q}_{\ell}}
\newcommand*{\Qlbarunits}{\Qlbar^\times}
\newcommand*{\R}{\mathbb{R}}
\newcommand*{\Z}{\mathbb{Z}}

\newcommand*{\BNpair}{(B,N)\text{-pair}}

\newcommand*{\dt}{\mathsf}
\newcommand*{\fsep}{\quad}

\newcommand*{\LocE}{\mathcal E}

\newcommand*{\I}{\mathrm{i}}
\newcommand*{\ii}{\mathfrak{i}}
\newcommand*{\jj}{\mathfrak{j}}

\newcommand*{\modulo}[2]{{\raisebox{.2em}{$#1$}/\raisebox{-.2em}{$#2$}}}
\newcommand*{\scA}{{\mathscr{A}}}

\newcommand*{\cC}{{\mathcal{C}}}
\newcommand*{\scD}{{\mathscr{D}}}
\newcommand*{\cE}{{\mathcal{E}}}
\newcommand*{\scF}{{\mathscr{F}}}
\newcommand*{\cH}{{\mathcal{H}}}
\newcommand*{\scH}{{\mathscr{H}}}
\newcommand*{\cL}{{\mathcal{L}}}
\newcommand*{\cM}{{\mathcal{M}}}
\newcommand*{\scM}{{\mathscr{M}}}
\newcommand*{\cN}{{\mathcal{N}}}
\newcommand*{\cO}{{\mathcal{O}}}
\newcommand*{\scW}{{\mathscr{W}}}

\DeclareMathOperator{\CF}{CF}
\DeclareMathOperator{\End}{End}

\DeclareMathOperator{\Hom}{Hom}
\DeclareMathOperator{\IC}{IC}
\DeclareMathOperator{\id}{id}
\DeclareMathOperator{\ind}{ind}

\DeclareMathOperator{\Int}{Int}

\DeclareMathOperator{\Irr}{Irr}

\DeclareMathOperator{\sgn}{sgn}

\DeclareMathOperator{\supp}{supp}

\DeclareMathOperator{\Trace}{Trace}
\DeclareMathOperator{\Uch}{Uch}
\title[Generalised Springer correspondence for type \texorpdfstring{$\dt E_8$}{E8}]{On the generalised Springer correspondence for groups of type \texorpdfstring{$\dt E_8$}{E8}}
\date{\today}
\author{Jonas Hetz}
\address{IDSR - Lehrstuhl für Algebra, Universität Stuttgart, Pfaffenwaldring 57, D--70569 Stuttgart, Germany}
\email{jonas.hetz@mathematik.uni-stuttgart.de}
\begin{document}

\begin{abstract}
We complete the determination of the generalised Springer correspondence for connected reductive algebraic groups, by proving a conjecture of Lusztig on the last open cases which occur for groups of type $\dt E_8$.
\end{abstract}

\maketitle

\section{Introduction}

Let $\mf G$ be a connected reductive algebraic group over an algebraic closure $k$ of the finite field $\Fp$ with $p$ elements, where $p$ is a prime. Let $\mf W$ be the Weyl group of $\mf G$ and $\cN_\mf G$ be the set of all pairs $(\cO,\cE)$ where $\cO\subseteq\mf G$ is a unipotent conjugacy class and $\cE$ is an irreducible local system on $\cO$ (taken up to isomorphism) which is equivariant for the conjugation action of $\mf G$. The Springer correspondence (originally defined in \cite{SCorr} for $p$ not too small; for arbitrary $p$ see \cite{LuGreenPol}) defines an injective map $\iota_\mf G\colon\Irr(\mf W)\hookrightarrow\cN_\mf G$ which plays a crucial role, for example, in the determination of the values of the Deligne--Lusztig Green functions of \cite{DL}; for recent surveys see \cite[Chap.~13]{DM20} and \cite[\S 2.8]{GM}. However, the map $\iota_\mf G$ is not surjective in general. In order to understand the missing pairs in $\cN_\mf G$, Lusztig \cite{LuIC} developed a generalisation of Springer's correspondence. This in turn constitutes a substantial part of the general problem of computing the complete character tables of finite groups of Lie type.

With very few exceptions, the problem of determining explicitly the generalised Springer correspondence has been solved in \cite{LuIC}, \cite{LuSp}, \cite{Sp} (see also the further references there). The exceptions occur for $\mf G$ of type $\dt E_6$ with $p\neq3$ and $\dt E_8$ with $p=3$. Recently, Lusztig \cite{LugenSpringer} settled the case where $\mf G$ is of type $\dt E_6$ and stated a conjecture concerning the last open cases in type $\dt E_8$. It is the purpose of this paper to prove that conjecture, thus completing the determination of the generalised Springer correspondence in all cases.

The paper is organised as follows. In \Cref{SecSpring}, we recall the definition of the generalised Springer correspondence, due to \cite{LuIC}. In \Cref{SecParam}, we explain two parametrisations of unipotent characters and unipotent character sheaves, one in terms of Lusztig's Fourier transform matrices associated to families in the Weyl group, and the other one in terms of Harish-Chandra series, for simple groups of adjoint type with a split rational structure. In the last two sections, we focus on the specific case where $\mf G$ is simple of type $\dt E_8$ in characteristic $p=3$: \Cref{SecWG2} contains the description of the generalised Springer correspondence for this group, so that we can formulate Lusztig's conjecture on the last open cases as indicated above. Finally, the proof of this conjecture is given in \Cref{SecPf}. It is based on considering the Hecke algebra associated to the finite group $\dt E_8(q)$ (where $q=3^f$ for some $f\geqslant1$) and its natural $\BNpair$, exploiting a well-known formula relating characters of this Hecke algebra with the unipotent principal series characters of $\dt E_8(q)$.

\begin{Not}
As soon as a prime $p$ is fixed in a given setting, we denote by $k=\Fpbar$ an algebraic closure of the finite field $\Fp$ with $p$ elements; furthermore, we tacitly assume to have fixed any other prime $\ell\neq p$, as well as an algebraic closure $\Qlbar$ of the field $\Ql$ of $\ell$-adic numbers. It will be convenient to assume the existence of an isomorphism $\Qlbar\simeq\C$ and to fix such an isomorphism once and for all\footnote{Strictly speaking, the existence of such an isomorphism requires the axiom of choice. However, what we really need is an isomorphism between algebraic closures of $\Q$ in $\Qlbar$ and $\C$, and such an isomorphism is known to exist without reference to the axiom of choice, cf. \cite[Remark 1.2.11]{DeligneWeyl2}.}, so that we can speak of \enquote{complex} conjugation or absolute values for elements of $\Qlbar$ using this isomorphism. In this way, we will also identify the rational numbers $\Q$ or the real numbers $\R$ as subsets of $\Qlbar$ and just write $\Q\subseteq\R\subseteq\Qlbar$. In several places below, we will need to fix a square root of $p$ (or of powers of $p$), so we do this right away:
\begin{equation}\label{sqrtp} 
\begin{split}
&\text{From now on we fix, once and for all, a square root } \sqrt p \text{ of } p \text{ in }\Qlbar. \\
&\text{Furthermore, whenever }q=p^e\; (e\geqslant1), \text{ we set } \sqrt q:={(\sqrt p)}^e.
\end{split}
\end{equation}
For any finite group $\Gamma$, we denote by $\CF(\Gamma)$ the set of class functions $\Gamma\rightarrow\Qlbar$ and by $\Irr(\Gamma)\subseteq\CF(\Gamma)$ the subset consisting of irreducible characters of $\Gamma$ over $\Qlbar$. Thus, $\Irr(\Gamma)$ is an orthonormal basis of $\CF(\Gamma)$ with respect to the scalar product
\[{\langle f, f'\rangle}_{\Gamma}:=|\Gamma|^{-1}\sum_{g\in\Gamma}f(g)\overline{f'(g)}\quad (\text{for }f, f'\in\CF(\Gamma)).\]
Now let $p$ be a prime and $\mf G$ be a connected reductive group over $k=\Fpbar$, defined over the finite subfield $\F_q\subseteq k$ where $q=p^f$ for some $f\geqslant1$, with corresponding Frobenius map $F\colon\mf G\rightarrow\mf G$. With respect to this $F$, we fix a maximally split torus $\mf T_0\subseteq\mf G$ and an $F$-stable Borel subgroup $\mf B_0\subseteq\mf G$ such that $\mf T_0\subseteq\mf B_0$. Let $\Phi$ be the set of roots of $\mf G$ with respect to $\mf T_0$, and let $\Pi\subseteq\Phi$ be the subset of simple roots determined by $\mf T_0\subseteq\mf B_0$. We denote by $\mf W=\modulo{N_\mf G(\mf T_0)}{\mf T_0}$ the Weyl group of $\mf G$ (relative to $\mf T_0$). Given any closed subgroup $\mf H\subseteq\mf G$ (including the case where $\mf H=\mf G$), we denote by $\mf H^\circ\subseteq\mf H$ its identity component, by $\mf Z(\mf H)\subseteq\mf H$ the centre of $\mf H$ and by $\mf H_{\mathrm{uni}}\subseteq\mf H$ the (closed) subvariety consisting of all unipotent elements of $\mf H$. If $F(\mf H)=\mf H$, we set
\[\mf H^F:=\{h\in\mf H\mid F(h)=h\}\subseteq\mf H,\]
a finite subgroup of $\mf H$. In particular, $\mf G^F$ is the finite group of Lie type associated to $(\mf G,F)$.
\end{Not}

\section{The generalised Springer correspondence}\label{SecSpring}

Let $p$ be a prime, and let $\mf G$ be a connected reductive group over $k=\Fpbar$, with Frobenius map $F\colon\mf G\rightarrow\mf G$ providing $\mf G$ with an $\Fq$-rational structure where $q=p^f$ for some $f\geqslant1$ (and with the further notation as in the introduction). In this section, we give the definition of the generalised Springer correspondence for $\mf G$, due to Lusztig \cite{LuIC}.

\begin{0}\label{IntroCS}
We start by briefly introducing some of the most important notions of Lusztig's theory of character sheaves \cite{LuIntroCS} and the underlying theory of perverse sheaves \cite{BBD}. For the details we refer to \cite{LuIC}, \cite{LuCS1}--\cite{LuCS5}. Recall that $\ell\neq p$ is a fixed prime. Let $\scD\mf G$ be the bounded derived category of constructible $\Qlbar$-sheaves on $\mf G$ in the sense of Beilinson--Bernstein--Deligne \cite{BBD}. To each $K\in\scD\mf G$ and each $i\in\Z$ is associated the $i$th cohomology sheaf $\scH^i(K)$, whose stalks $\scH^i_g(K)$ (for $g\in\mf G$) are finite-dimensional $\Qlbar$-vector spaces. The support of such a $K$ is defined as
\[\supp K:=\overline{\{g\in\mf G\mid \scH^i_g K\neq0\text{ for some }i\in\Z\}}\subseteq \mf G\]
(where the bar stands for the Zariski closure, here in $\mf G$). We define a sign
\begin{equation}\label{epsK}
\hat{\varepsilon}_K:=(-1)^{\dim\mf G-\dim\supp K}.
\end{equation}
Let $\scM\mf G$ be the full subcategory of $\scD\mf G$ consisting of the perverse sheaves on $\mf G$; the category $\scM\mf G$ is abelian. Assume that $X$ is a locally closed subvariety of $\mf G$ and that $\cL$ an irreducible $\Qlbar$-local system on $X$. (We will just speak of a \enquote{local system} when we mean a $\Qlbar$-local system from now on.) There is a unique extension of $\cL$ to $\overline X$, namely, the intersection cohomology complex $\IC(\overline X,\cL)[\dim X]$ (where $[\,]$ denotes shift), due to Deligne--Goresky--MacPherson (see \cite{GMP}, \cite{BBD}). The following notation will be convenient: For $X$ and $\cL$ as above, and for any closed subvariety $Y\subseteq\mf G$ such that $X\subseteq Y$, we denote by
\[{\IC(\overline X,\cL)[\dim X]}^{\#Y}\in\scD Y\]
the extension of $\IC(\overline X,\cL)[\dim X]$ to $Y$, by $0$ on $Y\setminus\overline X$. Now let us bring the Frobenius endomorphism $F\colon\mf G\rightarrow\mf G$ into the picture. Let $F^\ast\colon\scD\mf G\rightarrow\scD\mf G$ be its inverse image functor. Let $K\in\scD\mf G$ be such that $F^\ast K\cong K$ in $\scD\mf G$, and let $\varphi\colon F^\ast K\xrightarrow{\sim}K$ be an isomorphism. For each $i\in\Z$ and $g\in\mf G^F$, $\varphi$ induces linear maps
\[\varphi_{i,g}\colon\scH^i_g(K)\rightarrow\scH^i_g(K).\]
Since only finitely many of the $\scH^i_g(K)$ ($g\in\mf G^F$, $i\in\Z$) are non-zero, one can define a characteristic function \cite[8.4]{LuCS2}
\[\chi_{K,\varphi}\colon\mf G^F\rightarrow\Qlbar,\fsep g\mapsto\sum_{i\in\Z}(-1)^i\Trace(\varphi_{i,g},\scH^i_g(K)).\]
In \cite[\S 2]{LuCS1}, Lusztig defines the character sheaves on $\mf G$ as certain simple objects of $\scM\mf G$ which are equivariant for the conjugation action of $\mf G$ on itself. We denote by $\hat{\mf G}$ a set of representatives for the isomorphism classes of character sheaves on $\mf G$. An important subset of $\hat{\mf G}$ is the one consisting of \emph{cuspidal} character sheaves, as defined in \cite[3.10]{LuCS1}. We denote by ${\hat{\mf G}}^\circ\subseteq\hat{\mf G}$ a set of representatives for the isomorphism classes of cuspidal character sheaves on $\mf G$. The inverse image functor $F^\ast$ may also be regarded as a functor $F^\ast\colon\scM\mf G\rightarrow\scM\mf G$, and we have $F^\ast(\hat{\mf G})=\hat{\mf G}$. Thus, we can consider the subset
\[\hat{\mf G}^F:=\{A\in\hat{\mf G}\mid F^\ast A\cong A\}\subseteq\hat{\mf G}\]
consisting of the $F$-stable character sheaves in $\hat{\mf G}$. Of particular relevance for our purposes will be the set $\hat{\mf G}^{\mathrm{un}}\subseteq\hat{\mf G}$ of (representatives for the isomorphism classes of) \emph{unipotent} character sheaves. These are, by definition, the simple constituents of a perverse cohomology sheaf $\leftidx{^p}{\!H}{^i}(K_w^{\cL_0})\in\scM\mf G$ for some $w\in\mf W$ and some $i\in\Z$, where $K_w^{\cL_0}\in\scD\mf G$ is as defined in \cite[2.4]{LuCS1}, with respect to the trivial local system $\cL_0=\Qlbar$ on the torus $\mf T_0$.
\end{0}

\begin{0}\label{LuIC}
Let $\cN_\mf G$ be the set of all pairs $(\cO,\cE)$ where $\cO\subseteq\mf G$ is a unipotent conjugacy class and $\cE$ is an irreducible local system on $\cO$ (up to isomorphism) which is equivariant for the conjugation action of $\mf G$ on $\cO$. We define another set $\cM_\mf G$, as follows: Consider a triple $(\mf L,\cO_0,\cE_0)$ consisting of a Levi complement $\mf L$ of some parabolic subgroup of $\mf G$, a unipotent class $\cO_0$ of $\mf L$ and an irreducible local system $\cE_0$ on $\cO_0$ (up to isomorphism) which is equivariant for the conjugation action of $\mf L$ on $\cO_0$; furthermore, assume that $({\mf Z(\mf L)}^\circ.\cO_0,1\boxtimes\cE_0)$ is a cuspidal pair for $\mf L$ in the sense of \cite[2.4]{LuIC}, where $1\boxtimes\cE_0$ denotes the inverse image of $\cE_0$ under the canonical map ${\mf Z(\mf L)}^\circ.\cO_0\rightarrow\cO_0$. The group $\mf G$ acts naturally on the set of all such triples by means of $\mf G$-conjugacy:
\[\leftidx{^g}(\mf L,\cO_0,\LocE_0):=(g\mf Lg^{-1},g\cO_0g^{-1},\Int(g^{-1})^\ast\LocE_0)\quad\text{for }g\in\mf G,\]
where $\Int(g^{-1})\colon\mf G\rightarrow\mf G$ is the inner automorphism given by conjugation with $g^{-1}$. Let $\cM_\mf G$ be the set of equivalence classes of triples $(\mf L,\cO_0,\cE_0)$ as above under this action. By a slight abuse of notation, we will just write $(\mf L,\cO_0,\cE_0)\in\cM_\mf G$ rather than $[(\mf L,\cO_0,\cE_0)]\in\cM_\mf G$ or the like. We typically write $\ii,\ii',...$ for elements of $\cN_\mf G$ and $\jj,\jj',...$ for elements of $\cM_\mf G$. Following \cite[\S 6]{LuIC}, any $\jj\in\cM_\mf G$ gives rise to a certain perverse sheaf $K_\jj\in\scM\mf G$ whose endomorphism algebra $\scA_\jj:=\End_{\scM\mf G}(K_\jj)$ is isomorphic to the group algebra of the relative Weyl group $\scW_\jj:=W_\mf G(\mf L)=\modulo{N_\mf G(\mf L)}{\mf L}$, see \cite[Theorem 9.2]{LuIC}. Thus, the isomorphism classes of the simple direct summands of $K_\jj$ are naturally parametrised by $\Irr(\scW_\jj)$, so we have
\begin{equation}\label{DecompKj}
K_{\jj}\cong\bigoplus_{\phi\in\Irr(\scW_{\jj})}(A_\phi\otimes V_\phi)
\end{equation}
where $A_\phi$ is the simple direct summand of $K_{\jj}$ corresponding to $\phi\in\Irr(\scW_{\jj})$, and where $V_\phi=\Hom_{\scM\mf G}(A_\phi, K_{\jj})$. Then, for any $\phi\in\Irr(\scW_\jj)$, there exists a unique $(\cO,\cE)\in\cN_\mf G$ for which
\begin{equation}\label{AphiGuni}
A_\phi|_{\mf G_{\mathrm{uni}}}\cong{\IC(\overline{\cO}, \cE)[\dim{{\mf Z(\mf L)}^\circ}+\dim{\cO}]}^{\#\mf G_{\mathrm{uni}}},
\end{equation}
and the isomorphism class of $A_\phi$ is uniquely determined by this property among the simple perverse sheaves which are constituents of $K_{\jj}$, see \cite[24.1]{LuCS5}. So for each $\jj\in\cM_\mf G$, the above procedure gives rise to an injective map 
\[{\Irr(\scW_{\jj})}\hookrightarrow{\cN_\mf G}.\]
Conversely, given any $(\cO,\cE)\in\cN_\mf G$, there exists a unique $\jj\in\cM_\mf G$ such that $(\cO,\cE)$ is in the image of the map ${\Irr(\scW_{\jj})}\hookrightarrow{\cN_\mf G}$ just defined. So there is an associated surjective map
\[\tau\colon\cN_\mf G\rightarrow\cM_\mf G,\]
whose fibres are called the blocks of $\cN_\mf G$. For any $\jj=(\mf L,\cO_0,\cE_0)\in\cM_\mf G$, the elements in the block $\tau^{-1}(\jj)\subseteq\cN_\mf G$ are thus parametrised by the irreducible characters of $\scW_\jj=W_\mf G(\mf L)$. The collection of the bijections
\begin{equation}\label{GenSpring}
\Irr(\scW_\jj)\xrightarrow{\sim}\tau^{-1}(\jj)\qquad(\text{for }\jj\in\cM_\mf G)
\end{equation}
is called the generalised Springer correspondence. If $\ii=(\cO,\cE)\in\cN_\mf G$ and if $\phi\in\Irr(\scW_{\tau(\ii)})$ corresponds to $\ii$ under \eqref{GenSpring}, we will write $A_{\ii}:=A_\phi$. If we only consider the element $(\mf T_0, \{1\}, \Qlbar)\in\cM_\mf G$, the map \eqref{GenSpring} defines an injection
\begin{equation}\label{OrdSpring}
\Irr(\mf W)\hookrightarrow\cN_\mf G,
\end{equation}
which is called the (ordinary) Springer correspondence. The problem of determining the generalised Springer correspondence (that is, explicitly describing the bijections \eqref{GenSpring} for all $\jj\in\cM_\mf G$) can be reduced to considering simple algebraic groups $\mf G$ of simply connected type, thus can be approached by means of a case-by-case analysis. This has been accomplished for almost all such $\mf G$, thanks to the work of Lusztig \cite{LuIC}, Lusztig--Spaltenstein \cite{LuSp}, Spaltenstein \cite{Sp} (see also the references there for earlier results concerning the ordinary Springer correspondence), and again Lusztig \cite{LugenSpringer}, the only remaining open problems occur for $\mf G$ of type $\dt E_8$ in characteristic $p=3$ (for which a conjecture is made in \cite[\S 6]{LugenSpringer}). In particular, the ordinary Springer correspondence \eqref{OrdSpring} is explicitly known in complete generality.
\end{0}

\begin{0}\label{XiYi}
We keep the setting of \Cref{LuIC} and consider the Frobenius map $F\colon\mf G\rightarrow\mf G$. It defines actions on $\cN_\mf G$ and $\cM_\mf G$, given by
\[\cN_\mf G\rightarrow\cN_\mf G,\fsep(\cO,\cE)\mapsto(F^{-1}(\cO),F^\ast\cE),\]
and
\[\cM_\mf G\rightarrow\cM_\mf G,\fsep(\mf L,\cO_0,\cE_0)\mapsto(F^{-1}(\mf L),F^{-1}(\cO_0),F^\ast\cE_0).\]
Let $\cN_\mf G^F$, $\cM_\mf G^F$ be the respective sets of fixed points under these actions where, in terms of the local systems, this is only meant up to isomorphism, and for the triples in $\cM_\mf G$ in addition only up to $\mf G$-conjugacy. The map $\tau$ commutes with the action of $F$ on $\cN_\mf G$, $\cM_\mf G$, so it gives rise to a surjective map $\cN_\mf G^F\rightarrow\cM_\mf G^F$. Furthermore, the generalised Springer correspondence \eqref{GenSpring} induces bijections
\[{\Irr(\scW_\jj)}^F\xrightarrow{\sim}\tau^{-1}(\jj)\cap\cN_\mf G^F\qquad(\text{for }\jj\in\cM_\mf G^F).\]
(Here, ${\Irr(\scW_\jj)}^F$ denotes the set of all characters in $\Irr(\scW_\jj)$ which are invariant under the automorphism of $\scW_\jj$ induced by $F$.) Let $\jj=(\mf L,\cO_0,\cE_0)\in\cM_\mf G^F$ and $\phi\in{\Irr(\scW_{\jj})}^F$, and assume that $\ii=(\cO,\cE)$ is the corresponding element of $\tau^{-1}(\jj)\cap\cN_\mf G^F$. Choosing an isomorphism $\varphi_0\colon F^\ast\cE_0\xrightarrow{\sim}\cE_0$ which induces a map of finite order at the stalk of $\cE_0$ at any element of $\cO_0^F$ allows the definition of a unique isomorphism $\varphi_\ii\colon{F^\ast A_{\ii}}\xrightarrow{\sim}{A_{\ii}}$ (see \cite[24.2]{LuCS5} and also \Cref{CharFctLuAlg} below). We set
\begin{align*}
a_\ii&:=-\dim\cO-\dim{\mf Z(\mf L)}^\circ, \\
b_\ii&:=\dim\supp A_{\ii}, \\
d_{\ii}&:=\frac12(a_\ii+b_\ii).
\end{align*}
We then have
\[\scH^a(A_{\ii})|_{\cO}\cong\begin{cases} \cE&\text{if }a=a_\ii,\\ \hfil0&\text{if }a\neq a_\ii, \end{cases}\]
so we can define an isomorphism $\psi_{\ii}\colon{F^\ast\cE}\xrightarrow{\sim}{\cE}$ by the requirement that $q^{d_\ii}\psi_{\ii}$ is equal to the isomorphism ${F^\ast{\scH^{a_\ii}(A_{\ii})|_{\cO}}}\xrightarrow{\sim}{{\scH^{a_\ii}(A_{\ii})}|_{\cO}}$ induced by $\varphi_\ii$. It is shown in \cite[(24.2.4)]{LuCS5} that for any $u\in\cO^F$, the induced map $\psi_{{\ii},u}\colon\cE_u\rightarrow\cE_u$ on the stalk of $\cE$ at $u$ is of finite order. Now we define two functions
\[X_{\ii},\;Y_{\ii}\colon\mf G^F_{\mathrm{uni}}\rightarrow\Qlbar\]
by
\[X_{\ii}(u):=(-1)^{a_\ii}q^{-d_{\ii}}\chi_{A_{\ii},\varphi_\ii}(u)\]
and
\[Y_{\ii}(u):=\begin{cases}
\mathrm{Trace}(\psi_{{\ii},u},\cE_u)&\text{if }u\in \cO^F, \\
 \hfil0&\text{if }u\notin \cO^F,
\end{cases}\]
for $u\in\mf G^F_{\mathrm{uni}}$. Both $X_{\ii}$ and $Y_{\ii}$ are invariant under the conjugation action of $\mf G^F$ on $\mf G^F_{\mathrm{uni}}$.
\end{0}

\begin{Thm}[Lusztig {\cite[\S 24]{LuCS5}}]\label{LusztigAlgorithm}
In the setting of \Cref{XiYi}, the following hold.
\begin{enumerate}[(a)]
\item The functions $Y_{\ii}$, $\ii\in\cN_\mf G^F$, form a basis of the vector space consisting of all functions $\mf G^F_{\mathrm{uni}}\rightarrow\Qlbar$ which are invariant under $\mf G^F$-conjugacy.
\item There is a system of equations
\[X_{\ii}=\sum_{\ii'\in\cN_\mf G^F}p_{\ii',\ii}Y_{\ii'}\quad(\ii\in\cN_\mf G^F),\]
for some uniquely determined $p_{\ii',\ii}\in\Z$.
\end{enumerate}
\end{Thm}

\begin{proof}
See \cite[(24.2.7) and (24.2.9)]{LuCS5}. Note that the restrictions \cite[(23.0.1)]{LuCS5} on the characteristic $p$ of $k$ can be removed thanks to the remarks in \cite[3.10]{LuclCS}.
\end{proof}

\begin{0}\label{CorLusztigAlgorithm}
As an immediate consequence of \Cref{LusztigAlgorithm} (and of \Cref{XiYi}), we see that:
\begin{enumerate}
\item[(i)] We have $p_{\ii,\ii}=1$ for all $\ii\in\cN_\mf G^F$;
\item[(ii)] If $\ii'=(\cO',\cE')\neq\ii=(\cO,\cE)$, then $p_{\ii',\ii}\neq0$ implies $\cO'\neq\cO$ and $\cO'\subseteq\overline{\cO}$;
\item[(iii)] If $\ii',\ii\in\cN_\mf G^F$ belong to different blocks, we have $p_{\ii',\ii}=0$.
\end{enumerate}
So if we define a total order $\leqslant$ on $\cN_\mf G^F$ in such a way that for $\ii=(\cO,\cE)$, $\ii'=(\cO',\cE')\in\cN_\mf G^F$ we have
\[\ii'\leqslant \ii\;\text{ whenever }\;\cO'\subseteq\overline{\cO}\]
(note that the latter defines a partial order on the set of unipotent classes of $\mf G$), the matrix $(p_{\ii',\ii})_{\ii',\ii\in\cN_\mf G^F}$ has upper unitriangular shape. In \cite[\S 24]{LuCS5}, Lusztig provides an algorithm for computing this matrix $(p_{\ii',\ii})$ which entirely relies on combinatorial data. This algorithm is implemented in Michel's development version of {\sffamily {CHEVIE}} \cite{MiChv} and is accessible via the functions \texttt{UnipotentClasses} and \texttt{ICCTable}.
\end{0}

\begin{Rem}\label{AGu}
(a) Let $\cO\subseteq\mf G$ be a unipotent conjugacy class, and let us fix an element $u\in\cO$. Given a $\mf G$-equivariant irreducible local system $\cE$ on $\cO$, the stalk $\cE_{u}$ is in a natural way an irreducible module for the group $A_\mf G(u):=\modulo{C_\mf G(u)}{{C_\mf G(u)}^\circ}$, and the assignment $\cE\mapsto\cE_{u}$ defines a bijection between the set of isomorphism classes of $\mf G$-equivariant irreducible local systems on $\cO$ and the set of isomorphism classes of irreducible modules for $A_\mf G(u)$ (see \cite[3.5]{ShGeomOrb} or \cite[19.7]{LuCSDC4}). Using this identification, it will sometimes be convenient to write $(u,\varsigma)\in\cN_\mf G$ instead of $(\cO,\cE)$ when $\varsigma\in\Irr(A_\mf G(u))$ describes the local system $\cE$ and, similarly, $(\mf L,u,\varsigma)\in\cM_\mf G$ instead of $(\mf L,\cO_0,\cE_0)$ in case $u\in\cO_0$ and $\cE_0$ corresponds to $\varsigma\in\Irr(A_\mf L(u))$.

(b) We will also use the following notation: For $u\in\mf G_{\mathrm{uni}}$, we write $\overline u\in A_\mf G(u)$ for the image of $u$ under the canonical map $C_\mf G(u)\rightarrow A_\mf G(u)$. If $u\in\mf G_{\mathrm{uni}}$ is fixed and $A_\mf G(u)$ is a cyclic group generated by $\overline u$, we will denote the irreducible characters of $A_\mf G(u)$  just by their values at $\overline u$.

(c) Finally, assume that $\cO\subseteq\mf G$ is a unipotent conjugacy class so that $A_{\mf G}(u)=\langle\overline u\rangle$ for $u\in\cO$. If $\cE$ is an irreducible $\mf G$-equivariant local system on $\cO$ corresponding to the irreducible character of $A_\mf G(u)$ which takes the value $\zeta\in\Qlbarunits$ at $\overline u$, we will write $(\cO,\zeta):=(\cO,\cE)$. Similarly, if $(\mf L,\cO_0,\cE_0)\in\cM_\mf G$ and $A_\mf G(u)=\langle\overline u\rangle$ for $u\in\cO_0$, we set $(\mf L,\cO_0,\zeta):=(\mf L,\cO_0,\cE_0)$ if $\cE_0$ is parametrised by the irreducible character of $A_\mf L(u)$ which takes the value $\zeta$ at $\overline u$. (Note that this is independent of the choice of $u$ in either case, so we do not have to refer to a specific choice of $u\in\cO$, $u\in\cO_0$, respectively.)
\end{Rem}

\section{Parametrisations of unipotent characters and unipotent character sheaves}\label{SecParam}

Throughout this section, we assume that $\mf G$ is a simple algebraic group of adjoint type over $k=\Fpbar$ (for a prime $p$). We also assume that $F\colon\mf G\rightarrow\mf G$ is a Frobenius map with respect to an $\Fq$-rational structure on $\mf G$ (where $q$ is a power of $p$) such that $(\mf G, F)$ is of split type. The further notation is as in the introduction. We will summarise Lusztig's results showing that there are two \enquote{parallel} parametrisations of the unipotent characters of $\mf G^F$ and the unipotent character sheaves on $\mf G$, one in terms of Lusztig's families in the Weyl group as introduced in \cite[\S 4]{Luchars}, the other one via Harish-Chandra series. We mostly follow \cite[\S 3]{LuRCS}. Note that this is independent of the chosen $F$ (at least as long as $(\mf G, F)$ is of split type); see \cite{LuUchCatCen} for a conceptual explanation for this fact. At the end of this section, we provide precise descriptions of the characteristic functions of unipotent character sheaves and of the characteristic functions as they appear in \Cref{XiYi}, due to Lusztig \cite[\S 25]{LuCS5}, \cite[\S 24]{LuCS5}, respectively. This will be an essential ingredient in the proof of Lusztig's conjecture, see \Cref{SecPf} below.

\begin{0}\label{UchGF}
For $w\in\mf W$, let $R_w\in\CF(\mf G^F)$ be the virtual character of Deligne--Lusztig, as defined in \cite[\S 1]{DL}. We denote by
\[\Uch(\mf G^F):=\{\rho\in\Irr(\mf G^F)\mid{\langle\rho,R_w\rangle}_{\mf G^F}\neq0\text{ for some }w\in\mf W\}\subseteq\Irr(\mf G^F)\]
the set of unipotent characters of $\mf G^F$. Furthermore, given $\phi\in\Irr(\mf W)$, let
\[R_\phi:=\frac{1}{|\mf W|}\sum_{w\in\mf W}\phi(w)R_w\in\CF(\mf G^F)\]
be the unipotent almost character corresponding to $\phi$. In \cite[\S 4]{Luchars}, Lusztig introduces parameter sets for the irreducible characters of $\mf G^F$. Since we assumed that $F$ defines a split rational structure on $\mf G$ and since we will only be concerned with unipotent characters, this description simplifies considerably, so let us indicate it here (without providing the exact definitions). There is a set $X(\mf W)$, together with a pairing
\[\{\;,\;\}\colon X(\mf W)\times X(\mf W)\rightarrow\Qlbar,\]
as well as an explicit embedding
\begin{equation}\label{IrrWXW}
\Irr(\mf W)\hookrightarrow X(\mf W),\fsep \phi\mapsto x_\phi,
\end{equation}
all of which only depend on $\mf W$ (and not on $p$ or $q$). By \cite[Main Theorem 4.23]{Luchars}, there is a bijection
\begin{equation}\label{ParUchXW}
X(\mf W)\xrightarrow{\sim}\Uch(\mf G^F),\fsep x\mapsto\rho_x,
\end{equation}
satisfying
\begin{equation}\label{MultUch}
{\langle\rho_x,R_\phi\rangle}_{\mf G^F}=\Delta(x)\{x,x_\phi\}\quad\text{for any }x\in X(\mf W),\;\phi\in\Irr(\mf W).
\end{equation}
(Here, $\Delta(x)$ is a certain sign attached to $x\in X(\mf W)$, see \cite[4.14]{Luchars}.)
Furthermore, for any $x\in X(\mf W)$, Lusztig sets
\begin{equation}\label{DefRx}
R_x:=\sum_{y\in X(\mf W)}\{y,x\}\Delta(y)\rho_y\in\CF(\mf G^F),
\end{equation}
which we will refer to as the unipotent almost character corresponding to $x\in X(\mf W)$. We have $R_{x_\phi}=R_\phi$ for all $\phi\in\Irr(\mf W)$. Since the \enquote{Fourier matrix} $(\{x,y\})_{x,y\in X(\mf W)}$ is hermitian (in fact, it is symmetric with our assumptions) and of order $2$, we get, for any $x\in X(\mf W)$:
\begin{equation}\label{UchRx}
\rho_x=\Delta(x)\sum_{y\in X(\mf W)}\{y,x\}R_y.
\end{equation}
Thus, with respect to the scalar product ${\langle\;,\;\rangle}_{\mf G^F}$, the almost characters $R_x$, $x\in X(\mf W)$, form an orthonormal basis of the subspace of $\CF(\mf G^F)$ generated by $\Uch(\mf G^F)$, and knowing the values of the $R_x$ is equivalent to knowing the values of the unipotent characters of $\mf G^F$.
\end{0}

\begin{0}
Recall the notions on character sheaves introduced in \Cref{IntroCS}. In addition to those, we need the following. For $\phi\in\Irr(\mf W)$, let
\[R_\phi^{\cL_0}:=\frac1{|\mf W|}\sum_{w\in\mf W}\phi(w)\sum_{i\in\Z}{{(-1)}^{i+\dim\mf G}}\, {\leftidx{^p}{\!H}{^i}(K_w^{\cL_0})},\]
see \cite[14.10]{LuCS3}. For any $\phi\in\Irr(\mf W)$ and $A\in\hat{\mf G}^{\mathrm{un}}$, let $(A:R_\phi^{\cL_0})$ be the multiplicity of $A$ in $R_\phi^{\cL_0}$ in the subgroup of the Grothendieck group of $\scM\mf G$ spanned by the character sheaves \cite[(14.10.4)]{LuCS3}. The set $\hat{\mf G}^{\mathrm{un}}$ can be parametrised by means of a scheme analogous to that for the unipotent characters in \Cref{UchGF}: There is a parametrisation
\begin{equation}\label{ParUCSXW}
X(\mf W)\xrightarrow{\sim}\hat{\mf G}^{\mathrm{un}},\fsep x\mapsto A_x,
\end{equation}
such that
\begin{equation}\label{MultUCS}
(A_x:R_\phi^{\cL_0})=\hat\varepsilon_{A_x}\{x,x_\phi\}\quad\text{for any }x\in X(\mf W)\text{ and any }\phi\in{\Irr(\mf W)}.
\end{equation}
\end{0}

\begin{0}\label{LuRCS}
The conditions \eqref{MultUch}, \eqref{MultUCS} do not completely determine the parametrisations \eqref{ParUchXW}, \eqref{ParUCSXW}, respectively, so one needs further information to uniquely specify a labelling of $\Uch(\mf G^F)$ and $\hat{\mf G}^{\mathrm{un}}$. This is achieved by considering Harish-Chandra series and suitable \enquote{eigenvalues} associated with elements of $\Uch(\mf G^F)$ and $\hat{\mf G}^{\mathrm{un}}$. It will be convenient for our purposes to follow the description in \cite[\S 3]{LuRCS}. Let $W$ be an irreducible Weyl group (or the trivial group $\{1\}$) with simple reflections $S\subseteq W$. With each such $(W,S)$ is associated a small (possibly empty) set $\mathfrak S_W^\circ$ which in most cases may be regarded as a subset of the set of roots of unity in $\Qlbar$, see \cite[3.1]{LuRCS}. For $J\subseteq S$, let $W_J=\langle J\rangle$ be the Coxeter group with Coxeter generators $J$. The definition of the sets $\mathfrak S_W^\circ$ (with $W$ irreducible or $\{1\}$) shows that, if $\mathfrak S_{W_J}^\circ\neq\varnothing$, then $(W_J,J)$ is necessarily simple or $\{1\}$. Furthermore (still assuming that $\mathfrak S_{W_J}^\circ\neq\varnothing$), by \cite[Theorem 5.9]{LuCoxFrob}, one obtains another Weyl group $W^{S/J}$, with simple reflections corresponding to the elements in $S\setminus J$. We have $W^{S/\varnothing}=W$ and $W^{S/S}=\{1\}$. Let
\[\mathfrak S_W:=\bigl\{(J,\phi,\zeta)\;\big|\; J\subseteq S,\;\phi\in\Irr(W^{S/J}),\;\zeta\in\mathfrak S_{W_J}^\circ\bigr\}.\]
The sets $\mathfrak S_W^\circ$ and $\Irr(W)$ canonically embed into $\mathfrak S_W$, via
\[\mathfrak S_W^\circ\hookrightarrow\mathfrak S_W,\fsep\zeta\mapsto(S,1,\zeta),\]
and
\begin{equation}\label{IrrWSW}
\Irr(W)\hookrightarrow\mathfrak S_W,\fsep\phi\mapsto(\varnothing,\phi,1),
\end{equation}
respectively. (We have $\mathfrak S_{\{1\}}^\circ=\{1\}$.)
\end{0}

\begin{0}\label{ParUchHC}
In the setting of \Cref{LuRCS}, let us now take $W=\mf W$. We will typically denote by $J$ a subset of the simple roots $\Pi$ rather than the corresponding subset of the simple reflections in $\mf W$; the definition of $\mf W_J$, $\mf W^{\Pi/J}$, ... is analogous to that in \Cref{LuRCS}. For any $J\subseteq\Pi$, let $\mf P_J=\mf B_0\mf W_J\mf B_0$ be the standard parabolic subgroup of $\mf G$, and let $\mf L_J\subseteq\mf P_J$ be the unique Levi complement of $\mf P_J$ which contains $\mf T_0$. Let us denote by ${\Uch(\mf L_J^F)}^\circ\subseteq\Uch(\mf L_J^F)$ the set of cuspidal unipotent characters of $\mf G^F$. The definition of $\mathfrak S_{\mf W_J}$ is designed to be such that it coincides with the set of eigenvalues of Frobenius $\lambda_\rho$ for $\rho\in{\Uch(\mf L_J^F)}^\circ\subseteq\Uch(\mf L_J^F)$, as defined in \cite[Chap.~11]{Luchars}. More precisely, up to a few exceptions (which we shall not explicitly be concerned with here), a cuspidal unipotent character on $\mf L_J^F$ is uniquely determined by its eigenvalue of Frobenius, and this gives the bijection
\[{\Uch(\mf L_J^F)}^\circ\xrightarrow{\sim}\mathfrak S_{\mf W_J}^\circ\]
(see \cite[3.2, 3.3]{LuRCS} for the details). For any $J\subseteq\Pi$, let
\[R_{\mf L_J}^{\mf G}\colon\CF(\mf L_J^F)\rightarrow\CF(\mf G^F)\]
be the Harish-Chandra induction (obtained by inflating class functions on $\mf L_J^F$ to $\mf P_J^F$ via the canonical map $\mf P_J^F\rightarrow\mf L_J^F$ followed by inducing them from $\mf P_J^F$ to $\mf G^F$). For any (cuspidal) unipotent character $\rho_0$ of $\mf L_J^F$, the irreducible characters of $\mf G^F$ which appear as constituents of $R_{\mf L_J}^{\mf G}(\rho_0)$ are in $\Uch(\mf G^F)$; conversely, given any $\rho\in\Uch(\mf G^F)$, there exists a unique $J\subseteq\Pi$ and a unique $\rho_0\in{\Uch(\mf L_J^F)}^\circ$ such that $\rho$ is a constituent of $R_{\mf L_J}^{\mf G}(\rho_0)$, see \cite{Lurepchev}. Thus, $\Uch(\mf G^F)$ is partitioned into Harish-Chandra series, one for each pair $(\mf L_J,\rho_0)$ with $J\subseteq\Pi$ such that $\mathfrak S_{\mf W_J}^\circ\neq\varnothing$, and with $\rho_0\in{\Uch(\mf L_J^F)}^\circ$. As soon as a square root of $p$ (and of $q$) is fixed, as in \eqref{sqrtp}, the irreducible characters of $\mf W^{\Pi/S}$ naturally parametrise the unipotent characters of $\mf G^F$ in the series with respect to $(\mf L_J,\rho_0)$, see \cite[8.6, 8.7]{Luchars}. Hence, we obtain a bijection
\[\Uch(\mf G^F)\xrightarrow{\sim}\mathfrak S_\mf W.\]
\end{0}

\begin{Rem}\label{FixParUch}
As mentioned in \Cref{LuRCS}, the bijection $X(\mf W)\xrightarrow{\sim}\Uch(\mf G^F)$, $x\mapsto\rho_x$ in \eqref{ParUchXW} is in general not completely determined by the requirement \eqref{MultUch}. However, in view of \cite[Proposition 6.4]{DMParam}, it can be chosen in such a way that it satisfies the following two additional conditions, and then it is uniquely determined:
\begin{enumerate}
\item[(i)] We have $\lambda_{\rho_x}=\tilde\lambda_x$ for all $x\in X(\mf W)$, where $\lambda_{\rho_x}\in\Qlbarunits$ is the eigenvalue of Frobenius of $\rho_x$ as defined in \cite[Chap.~11]{Luchars}, and where $\tilde\lambda_x$ is defined in a way analogous to $\lambda(\overline x)$ in \cite[p.~135]{DMParam}.
\item[(ii)] We have $\rho_{x_\phi}=\rho_\phi$ for all $\phi\in\Irr(\mf W)$, where $x_\phi\in X(\mf W)$ is the image of $\phi$ under \eqref{IrrWXW}, and where $\rho_\phi\in\Uch(\mf G^F)$ denotes the image of $\phi$ under \eqref{IrrWSW} with $W=\mf W$. (Thus, $\rho_\phi$ is the unipotent principal series character parametrised by $\phi$.)
\end{enumerate}
In fact, up to a few \enquote{exceptional} $\phi\in\Irr(\mf W)$, (ii) is automatically implied by \eqref{MultUch}, see \cite[Proposition 12.6]{Luchars}. The \enquote{exceptional} $\phi$ are taken care of by (i) and the definition of the $\tilde\lambda_x$. Finally, (i) removes all other possible ambiguities.
\end{Rem}

\begin{0}\label{ParUCSHC}
Just as for the unipotent characters, there is also a parametrisation of the unipotent character sheaves $\hat{\mf G}^{\mathrm{un}}$ in terms of \enquote{Harish-Chandra series}, as follows. First of all, to any $A\in\hat{\mf G}$ is associated a root of unity $\lambda_A\in\Qlbar$ which may be defined as in \cite[3.6]{LuRCS} (see also \cite{Eft} or \cite[Theorem 3.3]{Sh1}). Now let $\mf P\subseteq\mf G$ be any parabolic subgroup, and let $\mf L\subseteq\mf P$ be any Levi complement of $\mf P$. There is an induction functor
\begin{equation}\label{indFunctor}
\ind_{\mf L\subseteq\mf P}^{\mf G}\colon\{\mf L\text{-equivariant perverse sheaves on } \mf L\}\rightarrow\scD\mf G.
\end{equation}
see \cite[4.1]{LuCS1}. Let $A_0\in{\hat{\mf L}}^{\circ,\mathrm{un}}$ be a cuspidal unipotent character sheaf. By the results of \cite[\S 4]{LuCS1}, the induced complex $\ind_{\mf L\subseteq\mf P}^\mf G(A_0)\in\scM\mf G$ is a semisimple perverse sheaf which is a direct sum of unipotent character sheaves on $\mf G$. Furthermore, the definition of $\ind_{\mf L\subseteq\mf P}^\mf G(A_0)$ turns out to be independent of $\mf L$ \cite[Proposition 4.5]{LuIC}, so we may (and will) just denote it as $\ind_{\mf L}^\mf G(A_0)$ from now on. On the other hand, any $A\in\hat{\mf G}^{\mathrm{un}}$ is a constituent of $\ind_{\mf L_J}^{\mf G}(A_0)$ for some $J\subseteq\Pi$ and some $A_0\in{\hat{\mf L}_J}^{\circ,\mathrm{un}}$. By \cite[\S 3]{LuIC} and \cite[Lemma 5.9]{Sh1}, the endomorphism algebra of any such $\ind_{\mf L}^\mf G(A_0)$ is canonically isomorphic to the group algebra of $W_\mf G(\mf L_J)\cong \mf W^{\Pi/J}$ (again assuming the fixed choices of $\sqrt p$, $\sqrt q$ in \eqref{sqrtp}). This gives rise to a parametrisation of $\hat{\mf G}^{\mathrm{un}}$ analogous to that of $\Uch(\mf G^F)$ described in \Cref{ParUchHC}: For any $J\subseteq\Pi$ such that $\mathfrak S_{\mf W_J}^\circ$ is non-empty, there is a bijection
\[\hat{\mf L}_J^{\circ,\mathrm{un}}\xrightarrow{\sim}\mathfrak S_{\mf W_J}^\circ,\]
which in most cases is defined by identifying a given $A\in\hat{\mf L}_J^{\circ,\mathrm{un}}$ with the root of unity $\lambda_A$. (There are a few exceptions where $A\in\hat{\mf L}_J^{\circ,\mathrm{un}}$ is not uniquely determined by $\lambda_A$ \cite[Theorem 3.7]{LuRCS}, but we will not be explicitly concerned with those.) In this way, the unipotent character sheaves on $\mf G$ are parametrised via
\[\hat{\mf G}^{\mathrm{un}}\xrightarrow{\sim}\mathfrak S_\mf W.\]
\end{0}

\begin{Cor}\label{Parcompatible}
Let $X(\mf W)\xrightarrow{\sim}\Uch(\mf G^F),\; x\mapsto\rho_x$ be the unique bijection which satisfies \eqref{MultUch} and the conditions (i) and (ii) in \Cref{FixParUch}. Then, with the notation of \Cref{ParUchHC}, \Cref{ParUCSHC}, there exists a unique bijection
\[X(\mf W)\xrightarrow{\sim}\hat{\mf G}^{\mathrm{un}},\fsep x\mapsto A_x,\]
which satisfies the property \eqref{MultUCS} and such that in addition the following diagram commutes (where $\Irr(\mf W)\hookrightarrow X(\mf W)$ is the embedding \eqref{IrrWXW}, and $\Irr(\mf W)\hookrightarrow\mathfrak S_{\mf W}$ is the embedding \eqref{IrrWSW} with $W=\mf W$):
\begin{center}
\begin{tikzcd}[column sep=12em, row sep=2em]
& X(\mf W)\arrow[dr, "\sim"]\arrow[dl, "\sim"'] &
\\
\Uch(\mf G^F)\arrow[dr, "\sim"'] & \Irr(\mf W)\arrow[d, hook]\arrow[u,hook] & \hat{\mf G}^{\mathrm{un}}\arrow[dl, "\sim"]
\\
& \mathfrak S_{\mf W} &
\end{tikzcd}
\end{center}
\end{Cor}

\begin{proof}
Clearly, the requirement on the commutativity of the diagram printed in the corollary uniquely specifies the bijection . The commutativity of the diagram with respect to $\Irr(\mf W)$ follows from \Cref{FixParUch}. The fact that the bijection $X(\mf W)\xrightarrow{\sim}\hat{\mf G}^{\mathrm{un}}$ thus defined satisfies \eqref{MultUCS} is noted in \cite[3.10]{LuRCS}.
\end{proof}

\begin{Rem}
The relevance of \Cref{Parcompatible} for our purposes is that the parametrisation of $\hat{\mf G}^{\mathrm{un}}$ in terms of $X(\mf W)$ can now directly be read off from the parametrisations of unipotent characters of $\mf G^F$ in terms of $\mathfrak S_\mf W$ (that is, Harish-Chandra series) and $X(\mf W)$. Thus, we can just use the tables provided in the appendix of \cite{Luchars} when referring to the parametrisation of unipotent character sheaves on $\mf G$.
\begin{assumption}
From now on, we assume that the bijections $X(\mf W)\xrightarrow{\sim}\Uch(\mf G^F)$, $x\mapsto\rho_x$, and $X(\mf W)\xrightarrow{\sim}\hat{\mf G}^{\mathrm{un}}$, $x\mapsto A_x$, are chosen according to \Cref{Parcompatible}. In particular, for a given $x\in X(\mf W)$, whenever we write $\rho_x\in\Uch(\mf G^F)$ or $A_x\in\hat{\mf G}^{\mathrm{un}}$, we tacitly refer to this convention.
\end{assumption}
\end{Rem}

\begin{Rem}\label{AxAi}
Recall the notation of \Cref{LuIC}. Let us consider an element $\jj=(\mf L,\cO_0,\cE_0)\in\cM_\mf G$. Since the elements of $\cM_\mf G$ are defined only up to $\mf G$-conjugacy, we may assume that $\mf L=\mf L_J$ is the standard Levi subgroup of the standard parabolic subgroup $\mf P_J\subseteq\mf G$ for some $J\subseteq\Pi$. Let $\Sigma:={\mf Z(\mf L_J)}^\circ.\cO_0$ and $\cE:=1\boxtimes\cE_0$, so that $(\Sigma,\cE)$ is a cuspidal pair for $\mf L_J$ in the sense of \cite[2.4]{LuIC}. We also set
\[A_0:={\IC(\overline{\Sigma},\cE)[\dim\Sigma]}^{\#\mf L_J}.\]
By \cite[(7.1.4)]{LuCS2} combined with \cite[Theorem 23.1]{LuCS5} and the results of \cite{LuclCS}, $A_0$ is a cuspidal character sheaf on $\mf L_J$. Furthermore, by \cite[Proposition 4.5]{LuIC}, we have a canonical isomorphism
\[K_\jj\cong\ind_{\mf L_J}^{\mf G}(A_0).\]
Recall (\Cref{LuIC}) that the $A_\ii$ for $\ii\in\tau^{-1}(\jj)$ are defined as simple direct summands of $K_\jj$. On the other hand, by \cite[Proposition 4.8]{LuCS1}, the simple direct summands of $\ind_{\mf L_J}^{\mf G}(A_0)$ are character sheaves on $\mf G$, and $A_0$ is a \emph{unipotent} character sheaf on $\mf L_J$ if and only if one (or, equivalently, any) of these simple direct summands is a \emph{unipotent} character sheaf on $\mf G$, so the analogous statement holds for the $A_\ii$ with $\ii\in\tau^{-1}(\jj)$. Let us now assume that $A_0$ is a unipotent character sheaf on $\mf L_J$. Hence, for any $\ii\in\tau^{-1}(\jj)$, $A_\ii$ is a unipotent character sheaf on $\mf G$, so there exists some $x\in X(\mf W)$ such that $A_\ii\cong A_x$. Comparing the descriptions in \Cref{LuIC}, \Cref{ParUCSHC} for the parametrisations of the simple constituents of $K_\jj$, $\ind_{\mf L_J}^{\mf G}(A_0)$, respectively, we see that they are entirely analogous. Thus, if $\ii$ is the image of $\phi\in\Irr(\scW_\jj)$ under \eqref{GenSpring}, then $A_x$ corresponds to $(J,\phi,\lambda_{A_0})\in\mathfrak S_\mf W$ in \Cref{ParUCSHC}, with the canonical identification $\scW_\jj\cong\mf W^{\Pi/J}$. In other words, considering the block $\tau^{-1}(\jj)$ with $\jj$ as above (so that the associated $A_0$ is unipotent), the problem of explicitly determining the bijection
\[\Irr(\scW_\jj)\xrightarrow{\sim}\tau^{-1}(\jj)\]
is equivalent to solving the following problem:
\begin{equation}\label{ReformGenSpring}
\text{For any }\ii\in\tau^{-1}(\jj),\text{ find the }x\in X(\mf W)\text{ such that }A_x\cong A_\ii.\tag{$\diamondsuit$}
\end{equation}
\end{Rem}

\begin{0}\label{CharFctNorm1}
Let $x\in X(\mf W)$, and let us consider the corresponding unipotent character sheaf $A_x\in\hat{\mf G}^{\mathrm{un}}$. Since we are assuming that $(\mf G,F)$ is of split type, it is known that $A_x$ is automatically $F$-stable. (This follows, for example, from \cite[\S 3]{LuRCS} or \cite{LuUchCatCen}; it is also implicitly contained in \cite[\S 5]{Sh1}, \cite[\S 4]{Sh2}.) Let $\varphi\colon F^\ast A_x\xrightarrow{\sim}A_x$ be an isomorphism. As $A_x$ is $\mf G$-equivariant for the conjugation action, the resulting characteristic function $\chi_{A_x,\varphi}$ lies in $\CF(\mf G^F)$. Here, $\varphi$ (and, hence, $\chi_{A_x,\varphi}$) is uniquely defined up to multiplication with a scalar in $\Qlbarunits$; however, explicitly specifying an isomorphism $\varphi$ appears to be a very difficult problem. In \cite[25.1]{LuCS5}, Lusztig formulates a condition which determines an isomorphism $F^\ast A_x\xrightarrow{\sim}A_x$ up to multiplication with a root of unity. So let us assume that $\varphi_x\colon F^\ast A_x\xrightarrow{\sim}A_x$ satisfies this requirement. Then, by the results of \cite[\S 25]{LuCS5}, the corresponding characteristic function $\chi_x:=\chi_{A_x,\varphi_x}\colon\mf G^F\rightarrow\Qlbar$ has norm $1$ with respect to the scalar product ${\langle\;,\;\rangle}_{\mf G^F}$; moreover, assuming that for any $x\in X(\mf W)$ a choice for $\varphi_x$ is made as above, the set of functions $\{\chi_x\mid x\in X(\mf W)\}$ is an orthonormal basis of the subspace of $\CF(\mf G^F)$ that it generates.
\end{0}

\begin{0}\label{CharFctLuAlg}
Let us consider an element $\ii\in\cN_\mf G^F$. We describe the isomorphism $\varphi_\ii\colon F^\ast A_\ii\xrightarrow{\sim}A_\ii$ in \Cref{XiYi}, following \cite[24.2]{LuCS5}. As in \Cref{AxAi}, we can represent $\jj=\tau(\ii)\in\cM_\mf G^F$ by a triple $(\mf L_J,\cO_0,\cE_0)$ for some $J\subseteq\Pi$. Thus, we have $F(\mf L_J)=\mf L_J$, and then also $F(\cO_0)=\cO_0$. Furthermore, $F$ acts trivially on $\scW_\jj=W_\mf G(\mf L_J)$. By \cite[\S 9]{LuIC}, there is a canonical basis $\theta_w$ ($w\in\scW_\jj$) for the endomorphism algebra $\scA_\jj=\End_{\scM\mf G}(K_\jj)$, and we have an isomorphism
\[\Qlbar[\scW_\jj]\xrightarrow{\sim}\scA_\jj,\fsep w\mapsto\theta_w.\]
Now $V_{A_\ii}:=\Hom_{\scM\mf G}(A_\ii,K_\jj)$ is an irreducible left $\scA_\jj$-module via the obvious composition of morphisms. By \cite[24.2]{LuCS5}, we can choose an isomorphism $\varphi_0\colon F^\ast\cE_0\xrightarrow{\sim}\cE_0$ which induces a map of finite order at the stalk of $\cE_0$ at any element of $\cO_0^F$, and such a choice determines an isomorphism $\varphi_\jj\colon F^\ast K_\jj\xrightarrow{\sim}K_\jj$. Let us first consider any isomorphism $\tilde\varphi_\ii\colon F^\ast A_\ii\xrightarrow{\sim}A_\ii$. This gives rise to a bijective linear map
\[\sigma_{\tilde\varphi_\ii}\colon V_{A_\ii}\rightarrow V_{A_\ii},\fsep v\mapsto\varphi_\jj\circ F^\ast(v)\circ\tilde\varphi_\ii^{-1}\]
which satisfies $\tilde\sigma_{A_\ii}\circ\theta_w=\theta_w\circ\tilde\sigma_{A_\ii}$ for all $w\in\scW_\jj$. (Here, for $v\colon A_\ii\rightarrow K_\jj$ in $\scM\mf G$, $F^\ast(v)$ denotes the induced map $F^\ast A_\ii\rightarrow F^\ast K_\jj$.) By Schur's lemma, there exists some $\zeta_\ii\in\Qlbarunits$ such that $\tilde\sigma_{\varphi_\ii}=\zeta_\ii\cdot\id_{V_{A_\ii}}$. Thus, setting $\varphi_\ii:=\zeta_\ii\tilde\varphi_\ii$, we have $\sigma_{\varphi_\ii}=\id_{V_{A_\ii}}$, and this requirement uniquely determines the isomorphism $\varphi_\ii\colon F^\ast A_\ii\xrightarrow{\sim}A_\ii$.
\end{0}

\begin{Rem}\label{FixIsoE0}
In the setting of \Cref{CharFctLuAlg}, we see that the isomorphism $\varphi_\ii\colon F^\ast A_\ii\xrightarrow{\sim}A_\ii$ is uniquely determined as soon as $\varphi_0\colon F^\ast\cE_0\xrightarrow{\sim}\cE_0$ is chosen. So it will be convenient for the further discussion to be able to refer to (any) fixed choices for these isomorphisms:
\begin{assumption}
For any $\jj=(\mf L_J,\cO_0,\cE_0)\in\cM_\mf G^F$, we assume to have fixed an isomorphism $\varphi_0\colon F^\ast\cE_0\xrightarrow{\sim}\cE_0$ which induces a map of finite order at the stalk of $\cE_0$ at any element of $\cO_0^F$.
\end{assumption}
Now let us assume that $A_\ii$ is a unipotent character sheaf on $\mf G$, so that there exists some $x\in X(\mf W)$ for which $A_x\cong A_\ii$. Then the isomorphism $\varphi_x$ in \Cref{CharFctNorm1} does in general not coincide with the isomorphism $\varphi_\ii$ defined in \Cref{CharFctLuAlg} (not even by multiplication with a root of unity)! Hence, the same holds for the associated characteristic functions.
\end{Rem}

\begin{Cor}\label{ChixXi}
Let $\ii=(\cO,\cE)\in\cN_\mf G^F$, and assume that $A_\ii$ is a unipotent character sheaf on $\mf G$. Let $x\in X(\mf W)$ be such that $A_x$ is isomorphic to $A_\ii$. Let $\jj=\tau(\ii)=(\mf L_J,\cO_0,\cE_0)\in\cM_\mf G^F$. Then there exists a unique isomorphism $\varphi_x\colon F^\ast A_x\xrightarrow{\sim}A_x$ as in \Cref{CharFctNorm1}, such that for any $u\in\mf G^F_{\mathrm{uni}}$, we have
\[\chi_{A_x,\varphi_{A_x}}(u)=q^{e_\ii}X_\ii(u)\]
where
\[e_\ii=\frac12(\dim\mf G-\dim\cO-\dim{\mf Z(\mf L)}^\circ).\]
\end{Cor}

\begin{proof}
We start with any isomorphism $\tilde\varphi_x\colon F^\ast A_x\xrightarrow{\sim}A_x$ as in \Cref{CharFctNorm1}. We have to compare such $\tilde\varphi_x$ with the isomorphism $\varphi_\ii\colon F^\ast A_\ii\xrightarrow{\sim}A_\ii$ defined in \Cref{CharFctLuAlg}. Recall from \Cref{CharFctLuAlg} that the fixed choice of $\varphi_0$ determines an isomorphism $\varphi_\jj\colon F^\ast K_\jj\xrightarrow{\sim}K_\jj$. As we have seen in \Cref{AxAi}, $A_x\cong A_\ii$ is a simple direct summand of $K_\jj$, so $V_{A_x}=\Hom_{\scM\mf G}(A_x,K_\jj)\cong V_{A_\ii}$ is an irreducible left $\scA_\jj$-module. Just as we did for $\varphi_{A_\ii}$, we may thus consider the bijective linear map
\[\sigma_{\tilde\varphi_x}\colon V_{A_x}\rightarrow V_{A_x},\fsep v\mapsto\varphi_\jj\circ F^\ast(v)\circ\tilde\varphi_x^{-1}.\]
As noted in \cite[p.~153]{LuCS5}, $\sigma_{\tilde\varphi_x}$ is equal to $q^{-(\dim\mf G-\dim\supp A_x)/2}$ times a map of finite order. We also know that there exists \emph{some} $\varpi_x\in\Qlbarunits$ for which $\tilde\varphi_x=\varpi_x\varphi_\ii$, so we get
\[\sigma_{\tilde\varphi_x}=\varpi_x^{-1}\sigma_{\varphi_\ii}=\varpi_x^{-1}\id_{V_{A_\ii}}.\]
Hence, $\delta_x:=q^{(\dim\mf G-\dim\supp A_x)/2}\varpi_x^{-1}$ must be a root of unity. So we get, for any $u\in\mf G^F_{\mathrm{uni}}$:
\[\chi_{A_x,\tilde\varphi_x}(u)=\varpi_x\chi_{A_\ii,\varphi_\ii}(u)=(-1)^{a_\ii}\delta_x^{-1}q^{(\dim\mf G-\dim\cO-\dim{\mf Z(\mf L)}^\circ)/2}X_\ii(u).\]
Then $\varphi_x:=(-1)^{a_\ii}\delta_x\tilde\varphi_x\colon F^\ast A_x\xrightarrow{\sim}A_x$ is the unique isomorphism for which the corollary holds. (Note that, as $(-1)^{a_\ii}\delta_x\in\Qlbarunits$ is a root of unity, the choice of $\varphi_x$ is in accordance with \Cref{CharFctNorm1}.)
\end{proof}

\begin{Rem}
Let $x\in X(\mf W)$ and $\ii\in\cN_\mf G^F$ be such that $A_x\cong A_\ii$. Note that the only choice that is involved to specify the isomorphisms $\varphi_\ii$ and $\varphi_x$ by means of \Cref{CharFctLuAlg} and \Cref{ChixXi} is the one of an isomorphism $F^\ast\cE_0\xrightarrow{\sim}\cE_0$ as in \Cref{FixIsoE0}. While both $\varphi_\ii$ and $\varphi_x$ will indeed change if we vary the choice of $\varphi_0\colon F^\ast\cE_0\xrightarrow{\sim}\cE_0$, the interrelation between $\varphi_\ii$ and $\varphi_x$ via \Cref{CharFctLuAlg} and \Cref{ChixXi} will not.
\end{Rem}

\section{Groups of type \texorpdfstring{$\dt E_8$}{E8} in characteristic \texorpdfstring{$3$}{3} and Lusztig's conjecture}\label{SecWG2}

From now until the end of the paper, we assume that $\mf G$ is the simple group of type $\dt E_8$ over $k=\overline{\mathbb F}_3$ and that $F\colon\mf G\rightarrow\mf G$ is the Frobenius map with respect to an $\Fq$-rational structure on $\mf G$, where $q=3^f$ for some $f\geqslant1$. Thus, $\mf G^F=\dt E_8(q)$ is necessarily the (untwisted) group $\dt E_8(3^f)$. The further notational conventions for this $\mf G$ and $F$ are as in the introduction; in addition, we use the following: We write $\Pi=\{\alpha_1,\alpha_2,\ldots,\alpha_8\}\subseteq\Phi$, labelled in such a way that the Dynkin diagram of $\mf G$ is given by
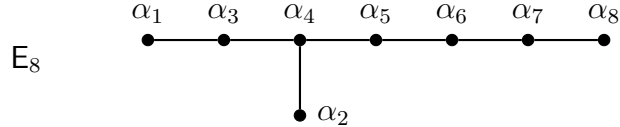
\begin{center}
\begin{tikzpicture}
    \draw (-1.25,-0.25) node[anchor=east]  {$\dt E_8$};

    \node[bnode,label=above:$\alpha_1$] 			(1) at (0,0) 	{};
    \node[bnode,label=right:$\alpha_2$] 			(2) at (2,-1) 	{};
    \node[bnode,label=above:$\alpha_3$] 			(3) at (1,0) 	{};
    \node[bnode,label=above:$\alpha_4$] 			(4) at (2,0) 	{};
    \node[bnode,label=above:$\alpha_5$] 			(5) at (3,0) 	{};
    \node[bnode,label=above:$\alpha_6$] 			(6) at (4,0) 	{};
	\node[bnode,label=above:$\alpha_7$] 			(7) at (5,0) 	{};
	\node[bnode,label=above:$\alpha_8$] 			(8) at (6,0) 	{};

    \path 	(1) edge[thick, sedge] (3)
          	(3) edge[thick, sedge] (4)
          	(4) edge[thick, sedge] (5)
			(4)	edge[thick, sedge] (2)
          	(5) edge[thick, sedge] (6)
			(6) edge[thick, sedge] (7)
			(7) edge[thick, sedge] (8);
\end{tikzpicture}
\end{center}
Viewing $\mf W$ as a reflection group associated to the root system $\Phi$, we denote by $s_i\in\mf W$ the reflection corresponding to the simple root $\alpha_i\in\Pi$. Thus, $\mf W$ is a Coxeter group with Coxeter generators $s_1,s_2,\ldots, s_8$, so that the products $s_1s_3$, $s_3s_4$, $s_2s_4$, $s_4s_5$, $s_5s_6$, $s_6s_7$, $s_7s_8$ have order $3$ while the product of any other two different Coxeter generators is of order $2$. Our notation for the unipotent conjugacy classes of $\mf G$ is the same as the one used by Spaltenstein \cite{Sp}.

\begin{0}\label{HCE8}
We recall the notation of \Cref{ParUchHC}. By the table in \cite[pp.~366--370]{Luchars} or \cite[pp.~484--488]{C}, we have $|X(\mf W)|=|\mathfrak S_\mf W|=166$. The set $\mathfrak S_{\mf W_J}^\circ$ is non-empty for the following subsets $J\subseteq\Pi$: $J=\varnothing$, $J=\{\alpha_2,\alpha_3,\alpha_4,\alpha_5\}$, $J=\{\alpha_1,\alpha_2,\ldots,\alpha_6\}$, $J=\{\alpha_1,\alpha_2,\ldots,\alpha_7\}$ and $J=\Pi$. Thus, the $166$ elements of $\mathfrak S_\mf W$ fall into the following Harish-Chandra series.
\begin{enumerate}
\item[(a)] For $J=\varnothing$, we have $112$ elements in the principal series, that is, they are in the image of the embedding $\Irr(\mf W)\hookrightarrow\mathfrak S_\mf W$, $\phi\mapsto(\varnothing,\phi,1)$.
\item[(b)] Let $J=\{\alpha_2,\alpha_3,\alpha_4,\alpha_5\}\subseteq\Pi$, so that the group $\modulo{\mf L_J}{{\mf Z(\mf L_J)}^\circ}$ is simple of type $\dt D_4$. We have $\mathfrak S_{\mf W_J}^\circ=\{-1\}$, and the relative Weyl group $\mf W^{\Pi/J}=W_\mf G(\mf L_J)$ is isomorphic to $W(\dt F_4)$. Thus, $\mf W^{\Pi/J}$ has $25$ irreducible characters, so there are $25$ elements in $\mathfrak S_\mf W$ of the form $(J,\phi,-1)$, $\phi\in\Irr(\mf W^{\Pi/J})$.
\item[(c)] Let $J=\{\alpha_1,\ldots,\alpha_6\}\subseteq\Pi$, so that the group $\modulo{\mf L_J}{{\mf Z(\mf L_J)}^\circ}$ is simple of type $\dt E_6$. We have $\mathfrak S_{\mf W_J}^\circ=\{\zeta_3,\zeta_3^2\}$ where $\zeta_3\in\Qlbarunits$ is a primitive $3$rd root of unity. The relative Weyl group $\mf W^{\Pi/J}=W_\mf G(\mf L_J)$ is isomorphic to $W(\dt G_2)\cong D_{12}$, the dihedral group of order $12$. Thus, $\mf W^{\Pi/J}$ has $6$ irreducible characters, so there are $6$ elements in $\mathfrak S_\mf W$ of the form $(J,\phi,\zeta_3)$, and $6$ elements of the form $(J,\phi,\zeta_3^2)$ (where $\phi\in\Irr(\mf W^{\Pi/J})$), in total $12$ elements.
\item[(d)] Let $J=\{\alpha_1,\ldots,\alpha_7\}\subseteq\Pi$, so that the group $\modulo{\mf L_J}{{\mf Z(\mf L_J)}^\circ}$ is simple of type $\dt E_7$. We have $\mathfrak S_{\mf W_J}^\circ=\{\I,-\I\}$ where $\I\in\Qlbarunits$ is a primitive $4$th root of unity. The relative Weyl group $\mf W^{\Pi/J}=W_\mf G(\mf L_J)$ is isomorphic to $W(\dt A_1)\cong\modulo{\Z}{2\Z}$. Thus, $\mf W^{\Pi/J}$ has $2$ irreducible characters $\pm1$, and $(J,\pm1,\pm\I)$ give rise to $4$ elements of $\mathfrak S_\mf W$.
\item[(e)] For $J=\Pi$, the set $\mathfrak S_{\mf W_J}^\circ=\mathfrak S_\mf W^\circ$ consists of $13$ elements, that is, $\mf G^F$ has $13$ cuspidal unipotent characters (and $\mf G$ has $13$ cuspidal unipotent character sheaves).
\end{enumerate}
\end{0}

\begin{0}\label{NGE8p3}
We now describe the set $\cN_\mf G$, following Lusztig \cite[15.3]{LuIC}. As the elements of $\cM_\mf G$ are defined only up to $\mf G$-conjugacy, we can (and will) always represent $\jj\in\cM_\mf G$ as an element of the form $\jj=(\mf L_J,\cO_0,\cE_0)$ for some $J\subseteq\Pi$. Note that, in contrast to the description of $\mathfrak S_\mf W$ in \Cref{HCE8}, the set $\cN_\mf G$ depends on the characteristic of $k$! With our assumption that $p=3$, we have $|\cN_\mf G|=127$, and the set $\cN_\mf G$ is partitioned into the following blocks $\tau^{-1}(\jj)$, $\jj\in\cM_\mf G$.
\begin{enumerate}
\item[(a)] Let $J=\varnothing$. There is a unique cuspidal pair on $\mf T_0$ in the sense of \cite[2.4]{LuIC}, namely, $(\mf T_0,\Qlbar)$. Thus, $\jj=(\mf T_0,\{1\},1)$ is uniquely determined by $J$, and it gives rise to the \enquote{Springer block} $\tau^{-1}(\jj)$.  We have $\scW_\jj=W_\mf G(\mf T_0)=\mf W$, and this group has $112$ irreducible characters. In view of \eqref{GenSpring}, we have $112$ elements in $\tau^{-1}(\jj)$ as well.
\item[(c)] Let $J=\{\alpha_1,\alpha_2,\ldots,\alpha_6\}$. According to \cite[\S 15]{LuIC}, there are two possibilities for $(\cO_0,\cE_0)$ such that $(\mf L_J,\cO_0,\cE_0)\in\cM_\mf G$. The relative Weyl group $\scW_\jj\cong W(\dt G_2)$ is the dihedral group of order $12$, so it has $6$ irreducible characters. Thus, each of the two $\jj$ above gives rise to $6$ elements of $\tau^{-1}(\jj)$, in total we get $12$ such elements of $\cN_\mf G$.
\item[(e)] Let $J=\Pi$. There are $3$ elements in $\cM_\mf G$ of the form $(\mf G,\cO_0,\cE_0)$; these are just the ones where $(\cO_0,\cE_0)$ is a cuspidal pair for $\mf G$ itself, and we obtain $3$ elements of $\cN_\mf G$ in this way.
\end{enumerate}
Note that in (a), the cuspidal character sheaf $A_0$ on $\mf T_0$ associated to $(\mf T_0,\{1\},1)$ in the setting of \Cref{AxAi} is just the constant local system $\cL_0=\Qlbar$ on $\mf T_0$, so it is in $\hat{\mf T}_0^{\mathrm{un}}$. The discussion in \Cref{AxAi} thus shows that we have
\[\{A_\ii\mid\ii\in\tau^{-1}(\jj)\}=\{A_x\mid x=x_\phi\text{ for some }\phi\in\Irr(\mf W)\}.\] Matching the elements in these two sets is the same as explicitly determining the ordinary Springer correspondence \eqref{OrdSpring}. This has been achieved by Spaltenstein \cite{Sp}. At the other extreme, let us consider the case (e). This tells us that there are $3$ cuspidal pairs $(\cO_0,\cE_0)$ for the group $\mf G$ itself, giving rise to $3$ cuspidal character sheaves on $\mf G$. These are just the complexes $K_\jj$ themselves, with $\jj$ of the form $(\mf G,\cO_0,\cE_0)$ as above, and we have $K_\jj\cong A_{(\cO_0,\cE_0)}$. So the set $\tau^{-1}(\jj)=\{(\cO_0,\cE_0)\}\subseteq\cN_\mf G$ is a singleton, the group $\scW_\jj$ is the trivial one, and, of course, determining the generalised Springer correspondence for these $\jj$ is evident. So it remains to consider the two cases where $J$ is as in (c), and this is what the rest of this paper will be concerned with.
\end{0}

\begin{Cor}\label{AiUnip}
For any $\ii\in\cN_\mf G$, the character sheaf $A_\ii$ is unipotent, that is, there exists some $x\in X(\mf W)$ such that $A_x$ is isomorphic to $A_\ii$. In particular, we have $\cN_\mf G=\cN_\mf G^F$. 
\end{Cor}

\begin{proof}
By the remarks in \Cref{NGE8p3}, the set of the $112$ different $A_\ii$ for $\ii\in\tau^{-1}((\mf T_0,\{1\},1))$ considered in \Cref{NGE8p3}(a) coincides with the set of the $A_{x_\phi}$ with $\phi\in\Irr(\mf W)$ considered in \Cref{HCE8}(a). For $\jj=(\mf G,\cO_0,\cE_0)$ as in \Cref{NGE8p3}(e), we have noted that $\ii=(\cO_0,\cE_0)$ is the unique element in its block $\tau^{-1}(\jj)$, and $A_\ii\cong K_\jj$ is a cuspidal character sheaf on $\mf G$. But any one of the $13$ cuspidal character sheaves on $\mf G$ is unipotent (see the proof of \cite[Proposition 5.3]{Sh2}), so $A_\ii$ is isomorphic to $A_x$ for some $x\in X(\mf W)$. It remains to consider the case where $\jj=(\mf L_J,\cO_0,\cE_0)$ is an element as in \Cref{NGE8p3}(c), so $J=\{\alpha_1,\alpha_2,\ldots,\alpha_6\}$. In view of \Cref{AxAi}, we have to show that the cuspidal character sheaf $A_0$ on $\mf L_J$ associated to $\jj$ is unipotent. Now $\modulo{\mf L_J}{{\mf Z(\mf L_J)}^\circ}$ is the simple (adjoint) group of type $\dt E_6$ in characteristic $3$, and the proof of \cite[Proposition 20.3]{LuCS4} shows that this group has exactly $2$ cuspidal character sheaves, both of which are unipotent. Using \cite[17.10]{LuCS4}, we conclude that $A_0\in{\hat{\mf L}_J}^\circ$ must be unipotent as well, as desired. Since all the unipotent character sheaves on $\mf G$ are $F$-stable (cf. \Cref{CharFctNorm1}), $\cN_\mf G=\cN_\mf G^F$ follows from \eqref{AphiGuni}.
\end{proof}

\begin{Rem}\label{AiAxE8p3}
\Cref{AiUnip} allows a \enquote{global} reformulation of \eqref{ReformGenSpring} in \Cref{AxAi} for $\mf G$ of type $\dt E_8$ in characteristic $3$:
\begin{equation}\label{ReformGenSpringE8p3}
\text{For any }\ii\in\cN_\mf G,\text{ find the }x\in X(\mf W)\text{ such that }A_x\cong A_\ii.\tag{$\diamondsuit'$}
\end{equation}
Thus, solving \eqref{ReformGenSpringE8p3} is equivalent to determining the generalised Springer correspondence \eqref{GenSpring} (with respect to all $\jj\in\cM_\mf G$) for $\mf G$.
\end{Rem}

\begin{0}\label{IntroE8p3E6}
We turn to the case \Cref{NGE8p3}(c). So let $J:=\{\alpha_1,\alpha_2,\ldots,\alpha_6\}\subseteq\Pi$, and let $(\cO_0,\cE_0)$ be one of the two pairs such that $\jj=(\mf L_J,\cO_0,\cE_0)\in\cM_\mf G$. From the proof of \cite[Proposition 20.3]{LuCS4}, we see that the two cuspidal (unipotent) character sheaves for simple groups of type $\dt E_6$ in characteristic $3$ are supported by the unipotent variety. In view of \cite[17.10]{LuCS4}, the two cuspidal (unipotent) character sheaves on $\mf L_J$ are thus supported by ${\mf Z(\mf L_J)}^\circ.(\mf L_J)_{\mathrm{uni}}$, so $\cO_0$ must be the regular unipotent class of $\mf L_J$. For $u_0\in\cO_0$, we have $A_{\mf L_J}(u_0)=\langle\overline u_0\rangle\cong\modulo{\Z}{3\Z}$. Hence, using the conventions in \Cref{AGu}, we can write $\jj=(\mf L_J,\cO_0,\cE_0)=(\mf L_J,\cO_0,\zeta)$ where $\zeta\in\Qlbarunits$ is a $3$rd root of unity. Recall from \Cref{AxAi} that $K_\jj$ is isomorphic to $\ind_{\mf L_J}^{\mf G}(A_0)$ where 
\[A_0:={\IC({\mf Z(\mf L_J)}^\circ.\overline{\cO}_0,1\boxtimes\cE_0)[\dim{\mf Z(\mf L_J)}^\circ+\dim\cO_0]}^{\#\mf L_J}\in\hat{\mf L}_J^{\circ,\mathrm{un}}.\]
We also recall (see \Cref{ParUCSHC}) that the isomorphism classes of the simple constituents of $\ind_{\mf L_J}^{\mf G}(A_0)$ are parametrised by triples $(J,\phi,\lambda_{A_0})\in\mathfrak S_\mf W$ where $\phi$ runs over the irreducible characters of $\mf W^{\Pi/J}$. The definition of $\lambda_{A_0}$ in \cite[3.6]{LuRCS} then shows that we have $\lambda_{A_0}=\zeta$, and this is a \emph{primitive} $3$rd root of unity by the definition of $\mathfrak S_{\mf W_J}^\circ$ in \cite[3.1]{LuRCS}. Now by \eqref{GenSpring}, we have a bijection
\[\Irr(\scW_\jj)\xrightarrow{\sim}\tau^{-1}(\jj).\]
Here, $\scW_\jj\cong\mf W^{\Pi/J}\cong W(\dt G_2)$ is the dihedral group of order $12$, so it has $6$ irreducible characters: four linear characters and two irreducible characters of degree $2$. While Spaltenstein \cite{Sp} has shown to which elements of $\tau^{-1}(\jj)$ the four linear characters of $\scW_\jj$ are mapped, this is not known for the two other irreducible characters! Put differently, using the characterisation \eqref{ReformGenSpring} in \Cref{AxAi}: For any of the two $x\in X(\mf W)$ corresponding to a triple $(J,\phi,\lambda_{A_0})\in\mathfrak S_\mf W$ with $\phi\in\Irr(\mf W^{\Pi/J})$ of degree $2$, the task is to find the $\ii\in\cN_\mf G$ such that $A_x\cong A_\ii$. (The two possibilities for those $\ii$ are $(\dt E_7+\dt A_1,\lambda_{A_0})$ and $(\dt E_7,\lambda_{A_0})$.) This ambiguity, for either of the two $A_0$ considered above, is exactly the last one that needs to be resolved to complete the generalised Springer correspondence. In order to describe it more precisely, we now proceed by fixing our notation for the relative Weyl group $\scW_\jj\cong\mf W^{\Pi/J}\cong W(\dt G_2)$ and its irreducible characters.
\end{0}

\begin{0}\label{WG2}
As before, let $J:=\{\alpha_1,\alpha_2,\ldots,\alpha_6\}\subseteq\Pi$. Let us give the concrete description of $\mf W^{\Pi/J}$, following \cite[3.1]{LuRCS}. The group $\mf W^{\Pi/J}$ is defined as the subgroup of $\mf W$ generated by the two elements
\[\sigma_i:=w_0^{J\cup\{\alpha_i\}}w_0^{J}=w_0^{J}w_0^{J\cup\{\alpha_i\}}\quad\text{for }i=7,8\]
(where, for $J'\subseteq\Pi$, $w_0^{J'}$ denotes the longest element of the Coxeter group $\mf W_{J'}$, the length function on $\mf W_{J'}$ being defined with respect to $J'\subseteq\Pi$). The elements $\sigma_7$, $\sigma_8$ have order $2$ and generate $\mf W^{\Pi/J}$ as a Coxeter group; more precisely, $\mf W^{\Pi/J}$ is a Weyl group coming from a natural root system induced from that in $\mf G$: This induced root system is of type $\dt G_2$, with $\sigma_7$ being the reflection with respect to a short simple root and $\sigma_8$ the reflection with respect to a long simple root, so the associated Dynkin diagram is as follows:
\begin{center}
\begin{tikzpicture}
    \draw (-1,0) node[anchor=east]  {$\dt G_2$};

    \node[bnode,label=above:$\sigma_7$] (1) at (0,0) {};
    \node[bnode,label=above:$\sigma_8$] (2) at (1,0) {};

    \path (2) edge[tedge] (1)
          ;
\end{tikzpicture}
\end{center}
Note that this is compatible with \cite[Theorem 9.2(a)]{LuIC}: We have a canonical isomorphism $\mf W^{\Pi/J}\cong W_\mf G(\mf L_J)$ under which $\sigma_i$ corresponds to the unique non-trivial element of $\modulo{N_{\mf L_{J\cup\{\alpha_j\}}}(\mf L_J)}{\mf L_J}$ for $i=7,8$ (cf. \cite[2.1]{LugenSpringer}). We now give our names for the irreducible characters of $\mf W^{\Pi/J}\cong W_\mf G(\mf L_J)$, which will be similar to those used in \cite[5.1]{LugenSpringer} (except that we write $\sgn$ instead of $s$): Thus, $1$ denotes the trivial character, $\sgn$ the sign character; $\epsilon, \epsilon'$ are the two remaining irreducible characters of degree $1$ such that
\[\epsilon(\sigma_7)=\epsilon'(\sigma_8)=1,\quad\epsilon(\sigma_8)=\epsilon'(\sigma_7)=-1.\]
Furthermore, $\rho$ is the character of the reflection representation and $\rho'$ the other irreducible character of degree $2$.
\end{0}

\begin{0}
We still assume that $J=\{\alpha_1,\alpha_2,\ldots,\alpha_6\}\subseteq\Pi$. Since we will refer to different sources in several places below, we give the explicit conversion between our notation concerning the relative Weyl group $\mf W^{\Pi/J}\cong W_\mf G(\mf L_J)$ in \Cref{WG2} and the ones used by Lusztig \cite[p.~361]{Luchars}, Carter \cite[p.~412]{C} (which coincides with the one in Michel's {\sffamily CHEVIE} \cite{MiChv}) and Spaltenstein \cite[p.~327]{Sp}, respectively. First of all, it is easy to compare our notation with the one in Lusztig's book as he explicitly states his conventions in \cite[p.~361]{Luchars}. The fact that our short (respectively, long) roots coincide with the short (respectively, long) roots as indicated by Carter \cite[p.~412]{C} implicitly follows from comparing his table \cite[p.~487]{C} describing the Harish-Chandra series of unipotent characters with Lusztig's \cite[p.~367]{Luchars}. However, one must be extremely careful when translating Spaltenstein's conventions to our notation as far as the characters $\epsilon$, $\epsilon'$ are concerned: Note that Spaltenstein's declarations for the long and short roots \cite[p.~327]{Sp} are opposite to ours (hence also to Lusztig's and Carter's)! Thus, forgetting the underlying root system and just regarding $\mf W^{\Pi/J}$ as a Coxeter group with Coxeter generators $\sigma_7,\sigma_8$ for a moment: In order to match with the setting of \cite{Sp}, we must identify $\sigma_7\leftrightarrow s_\ell$ and $\sigma_8\leftrightarrow s_c$, and then we have $\epsilon\leftrightarrow\epsilon_c$, $\epsilon'\leftrightarrow\epsilon_\ell$. (The other irreducible characters are not affected by the distribution of long or short roots.) The conversion scheme between our notation and that of Lusztig, Carter (thus also {\sffamily CHEVIE}) and Spaltenstein as described above is given in \Cref{TblConvG2CarLusSp}. (Whenever we refer to \cite{Sp}, we will still explicitly provide the relevant results with our notation to avoid any possible confusion, most notably the table in \cite[p.~329]{Sp}, see \Cref{TblSpaltE8E6} below.) Note that Spaltenstein does not distinguish the characters $\theta'$, $\theta''$ (which is the last ambiguity in the generalised Springer correspondence that remains to be resolved), so we only know that $\{\theta'$, $\theta''\}\leftrightarrow\{\rho,\rho'\}$ at this point.
{\renewcommand{\arraystretch}{1.4}
\begin{table}[htbp]
\centering
\begin{tabular}{c|c|c|c}
\hline
Our notation & Lusztig \cite[p.~361]{Luchars} & Carter \cite[p.~412]{C} & Spaltenstein \cite[p.~327]{Sp} \\
\hline
$1$ & $1$ & $\phi_{1,0}$ & $1$ \\
\hline
$\sgn$ & $\varepsilon$ & $\phi_{1,6}$ & $\varepsilon$ \\
\hline
$\epsilon$ & $\varepsilon'$ & $\phi_{1,3}'$ & $\varepsilon_c$ \\
\hline
$\epsilon'$ & $\varepsilon''$ & $\phi_{1,3}''$ & $\varepsilon_\ell$ \\
\hline
$\rho$ & $r$ & $\phi_{2,1}$ & $\theta'$ or $\theta''$ \\
\hline
$\rho'$ & $r'$ & $\phi_{2,2}$ & $\theta''$ or $\theta'$
\end{tabular}
\caption{Conversion scheme between our notation for the irreducible characters of $\mf W^{\Pi/J}\cong W(\dt G_2)$ and the ones used in \cite[p.~412]{C}, \cite[p.~361]{Luchars} and \cite[p.~327]{Sp}}
\label{TblConvG2CarLusSp}
\end{table}
}
\end{0}

\begin{0}\label{IntroLuConj}
Let $J=\{\alpha_1,\alpha_2,\ldots,\alpha_6\}$. Continuing our discussion in \Cref{IntroE8p3E6}, we can now state Lusztig's conjecture in a precise way. So let $\jj$ be one of the two elements of $\cM_\mf G$ of the form $(\mf L_J,\cO_0,\zeta_3)$ where $\zeta_3\in\Qlbarunits$ is a primitive $3$rd root of unity. With the notation of \Cref{WG2}, the elements in the block $\tau^{-1}(\jj)\subseteq\cN_\mf G$ are parametrised by
\[\Irr(\mf W_\mf G(\mf L_J))=\Irr(\mf W^{\Pi/J})=\{1,\sgn,\epsilon,\epsilon',\rho,\rho'\}.\]
As already mentioned before, Spaltenstein \cite{Sp} has shown to which element $\ii\in\tau^{-1}(\jj)$ any of the characters $1,\sgn,\epsilon,\epsilon'$ belongs, but he has not given the correspondence as far as $\rho,\rho'$ are concerned. We will for now use the notation in \cite{Sp} for the last two characters, that is, we just write $\{\theta',\theta''\}=\{\rho,\rho'\}$ without yet distinguishing between $\theta'$ and $\theta''$. With this ambiguity, the generalised Springer correspondence for the two blocks is provided in the table in \cite[p.~329]{Sp}. The conversion to our notation is given in \Cref{TblSpaltE8E6}; this table must be understood as follows: The first column contains the unipotent classes $\cO$ for which there exists a local system $\cE$ on $\cO$ such that $(\cO,\cE)\in\tau^{-1}(\jj)$. The second column gives the group $A_\mf G(u)=\modulo{C_\mf G(u)}{C_\mf G^\circ(u)}$ for $u\in\cO$. Here, we have $A_\mf G(u)=\langle\overline u\rangle$ in all cases, so the local system on $\cO$ can be described by the value of the corresponding irreducible character of $A_\mf G(u)$ at $\overline u$ in the third column; this value is in fact equal to $\zeta_3$ for any one of the unipotent classes considered. The last column gives the name of the irreducible character of $\mf W_\mf G(\mf L_J)\cong\mf W^{\Pi/J}$ (notation of \Cref{WG2}) which is mapped to $(\cO,\zeta_3)$ under the generalised Springer correspondence. With the distribution of $\theta'$, $\theta''$ as in \Cref{TblSpaltE8E6}, Lusztig's conjecture \cite[\S 6]{LugenSpringer} states:
\begin{equation}\label{OpenCaseE8p3LuConj}
\text{We have }\theta'=\rho\;\text{ and }\;\theta''=\rho',\;\text{for any one of the two }\jj\in\cM_\mf G\text{ as above.}\tag{$\heartsuit$}
\end{equation}
{\renewcommand{\arraystretch}{1.4}
\begin{table}[htbp]
\centering
\begin{tabular}{c|c|c|c}
Class of $u$ & $A_\mf G(u)$ & $\varsigma$ & $\rho_{u,\phi}^{\mf G}$ \\
\hline
$\dt E_8$ & $\modulo{\Z}{3\Z}$ & $\zeta_3$ & $1$ \\
\hline
$\dt E_8(a_1)$ & $\modulo{\Z}{3\Z}$ & $\zeta_3$ & $\epsilon$ \\
\hline
$\dt E_7+\dt A_1$ & $\modulo{\Z}{6\Z}$ & $\zeta_3$ & $\theta'$ \\
\hline
$\dt E_7$ & $\modulo{\Z}{3\Z}$ & $\zeta_3$ & $\theta''$ \\
\hline
$\dt E_6+\dt A_1$ & $\modulo{\Z}{3\Z}$ & $\zeta_3$ & $\epsilon'$ \\
\hline
$\dt E_6$ & $\modulo{\Z}{3\Z}$ & $\zeta_3$ & $\sgn$
\end{tabular}
\caption{The (incomplete) generalised Springer correspondence in $\dt E_8$, $p=3$, with respect to the blocks $\tau^{-1}(\jj)$, $\jj=(\mf L_J,\cO_0,\zeta_3)$, where $\cO_0\subseteq\mf L_J$ is the regular unipotent class and $\zeta_3\in\Qlbarunits$ is a primitive $3$rd root of unity, due to Spaltenstein \cite[p.~329]{Sp}.}
\label{TblSpaltE8E6}
\end{table}
}
\end{0}

\section{The proof of Lusztig's conjecture}\label{SecPf}

The setting and notation is exactly the same as in the previous section. (In particular, $\mf G$ is simple of type $\dt E_8$, $p=3$ and $\mf G^F=\dt E_8(3^f)$, $f\geqslant1$.) In addition, the names for the irreducible characters of $\mf W$ are as in \cite[4.13]{Luchars}. From now until the end of the paper, we use the following convention (see \Cref{SecParam}):
\begin{assumption}
For any $\jj=(\mf L_J,\cO_0,\cE_0)\in\cM_\mf G^F$, we assume to have fixed an isomorphism $\varphi_0\colon F^\ast\cE_0\xrightarrow{\sim}\cE_0$ which induces a map of finite order at the stalk of $\cE_0$ at any element of $\cO_0^F$. Then, for any $\ii\in\tau^{-1}(\jj)\subseteq\cN_\mf G^F=\cN_\mf G$, with corresponding $x\in X(\mf W)$ such that $A_x\cong A_\ii$ (see \Cref{AiUnip}), we set $\chi_\ii:=\chi_{A_\ii,\varphi_\ii}$ and $\chi_x:=\chi_{A_x,\varphi_x}$, where $\varphi_\ii\colon F^\ast A_\ii\xrightarrow{\sim}A_\ii$ and $\varphi_x\colon F^\ast A_x\xrightarrow{\sim}A_x$ are the unique isomorphisms defined through \ref{CharFctLuAlg} and \ref{ChixXi}, respectively.
\end{assumption}

\begin{0}\label{Hecke}
Let us consider the Hecke algebra associated to $\mf G^F$ and its $\BNpair$ $(\mf B_0^F,N_{\mf G}(\mf T_0)^F)$, that is, the endomorphism algebra
\[\cH_q:=\End_{\Qlbar\mf G^F}\bigl(\Qlbar[{\mf G^F}/{\mf B_0^F}]\bigr)^\mathrm{opp}.\]
(Here, \enquote{opp} stands for the opposite algebra.) $\cH_q$ has a $\Qlbar$-basis $\{T_w\mid w\in\mf W\}$ where
\[T_w\colon\Qlbar[{\mf G^F}/{\mf B_0^F}]\rightarrow\Qlbar[{{\mf G}^F}/{\mf B_0^F}],\quad x\mf B_0^F\mapsto\sum_{\substack{y\mf B_0^F\in{\mf G^F}/{\mf B_0^F} \\ x^{-1}y\in\mf B_0^F\dot w \mf B_0^F}}y\mf B_0^F,\]
for $w\in\mf W$. (Here and in what follows, whenever $w\in\mf W$ is given, we denote by $\dot w$ a representative of $\mf W$ in ${N_\mf G(\mf T_0)}^F$.) Let $\ell\colon\mf W\rightarrow\mathbb Z_{\geqslant0}$ be the length function of $\mf W$ with respect to the Coxeter generators $S=\{s_1,\ldots,s_8\}$. Then the multiplication in $\cH_q$ is determined by the following equations:
\[T_s\cdot T_w=\begin{cases}\hfil T_{sw}\quad&\text{if}\quad\ell(sw)=\ell(w)+1 \\ qT_{sw}+(q-1)T_{w}\quad&\text{if}\quad\ell(sw)=\ell(w)-1\end{cases}\quad(\text{for }s\in S,\, w\in\mf W).\]
For any $g\in\mf G^F$ and any $w\in\mf W$, we set
\[m(g,w):=\frac{|O_g\cap\mf B_0^F\dot w\mf B_0^F|\cdot|C_{\mf G^F}(g)|}{|\mf B_0^F|}\]
where $O_g\subseteq\mf G^F$ is the $\mf G^F$-conjugacy class of $g$. The irreducible characters of $\mf W$ naturally parametrise the isomorphism classes of irreducible modules of $\cH_q$, see \cite[Cor. 8.7]{Luchars} (recall our fixed choice of $\sqrt q$ in \eqref{sqrtp}). Given $\phi\in\Irr(\mf W)$, let $V_\phi$ be the corresponding module of $\cH_q$, and let $\rho_\phi\in\Uch(\mf G^F)$ be the image of $\phi$ under the map $\Irr(\mf W)\hookrightarrow\mathfrak S_\mf W\xrightarrow{\sim}\Uch(\mf G^F)$, see \Cref{ParUchHC}. By \cite[Remark 3.6(b)]{GCH} and \cite[\S 11D]{CR1}, we have
\[m(g,w)=\sum_{\phi\in\Irr(\mf W)}\rho_\phi(g)\Trace(T_w,V_\phi)\quad\text{for any }g\in\mf G^F,\;w\in\mf W,\]
where $\Trace(T_w,V_\phi)$ is the trace of the endomorphism of $V_\phi$ defined by $T_w$. It is shown in \cite[Corollary 14.14]{LuCS3} (see also \cite[3.6]{Luvaluni}) that the unipotent principal series characters $\rho_\phi$, $\phi\in\Irr(\mf W)$, can be expressed as linear combinations of characteristic functions of unipotent character sheaves: There exists a root of unity $\xi_x\in\Qlbarunits$ such that
\[\rho_\phi=(-1)^{\dim\mf G}\sum_{x\in X(\mf W)}\xi_x(A_x:R_\phi^{\cL_0})\chi_x\quad\text{for }\phi\in\Irr(\mf W).\]
(Note that, in general, $\xi_x$ depends on the choice of the isomorphism $F^\ast A_x\xrightarrow{\sim}A_x$, as well as on the one of a square root of $q$ in $\Qlbar$, but as we have fixed those in the beginning of this section and in \eqref{sqrtp}, respectively, we just write $\xi_x$ here. Also recall that our choice of $\varphi_x$ meets the requirement in \cite[25.1]{LuCS5}, which is the same as \cite[(13.8.1)]{LuCS3}.) Using \eqref{MultUCS}, we get
\[\rho_\phi=(-1)^{\dim\mf G}\sum_{x\in X(\mf W)}\xi_x\hat\varepsilon_x\{x,x_\phi\}\chi_x\quad\text{for }\phi\in\Irr(\mf W),\]
where $\hat\varepsilon_x:=\hat\varepsilon_{A_x}\in\{\pm1\}$. For any $g\in\mf G^F$ and any $w\in\mf W$, we thus obtain
\begin{equation}\label{GCarpath}
m(g,w)=\sum_{x\in X(\mf W)}\Biggl(\sum_{\phi\in\Irr(\mf W)}\{x,x_\phi\}\Trace(T_w,V_\phi)\Biggr)\nu_x\cdot \chi_x(g),\tag{$\spadesuit$}
\end{equation}
with
\[\nu_x:=(-1)^{\dim\mf G}\xi_x\hat\varepsilon_x\in\Qlbarunits\qquad\text{(a root of unity).}\]
The character values $\Trace(T_w,V_\phi)$ of the Hecke algebra $\cH_q$ have been determined by Geck and Michel in \cite{GeMi}; they are (just as the entries $\{x,x_\phi\}$ of the Fourier matrix) readily available through  {\sffamily {CHEVIE}} \cite{MiChv}, using the function \texttt{HeckeCharValues}. Hence, the coefficients of $\nu_x\chi_x(g)$ in \eqref{GCarpath} can be computed explicitly.
\end{0}

\begin{0}
Let $\ii\in\cN_\mf G^F=\cN_\mf G$. Thus, by \Cref{ChixXi}, whenever we know the $x\in X(\mf W)$ corresponding to $\ii$ under \eqref{ReformGenSpringE8p3} in \Cref{AiAxE8p3}, we can express $\chi_x|_{\mf G^F_\mathrm{uni}}$ as an \emph{explicit} scalar multiple of $X_\ii$. There are in total $4$ different $\ii\in\cN_\mf G$ for which the solution to \eqref{ReformGenSpringE8p3} is not known, see \Cref{IntroLuConj}. The idea to prove \eqref{OpenCaseE8p3LuConj} in \Cref{IntroLuConj} consists in evaluating the equation \eqref{GCarpath} in \Cref{Hecke} with a suitable $w\in\mf W$ (and $u\in\mf G^F_{\mathrm{uni}}$) and obtaining contradictions whenever \eqref{OpenCaseE8p3LuConj} does not hold. (We will not have to consider such $x\in X(\mf W)$ which do not correspond to an element of $\cN_\mf G$ as above, see \Cref{NonVanishRxNG} below.) Let us set
\begin{equation}\label{Defw49}
w_{49}:=s_2s_3s_4s_3s_5s_4s_1s_2s_3s_4s_5s_6s_7s_8\in\mf W
\end{equation}
(i.e., the $49$th element in \texttt{ChevieClassInfo($\mf W$).classtext} in {\sffamily CHEVIE} \cite{MiChv}). Then $w_{49}$ is an element of minimal length in its conjugacy class $C_{49}\subseteq\mf W$. In \cite{LuWeylUni}, Lusztig defined a surjective map from $\mf W$ to the set of unipotent classes of $\mf G$, and $C_{49}$ is sent to $\dt E_7+\dt A_1$ under this map. By the results of \cite{LuWeylUni} and \cite[4.8]{LuEllip}, we have $\cO\cap\mf B_0\dot w_{49}\mf B_0=\varnothing$ unless $\cO\in\{\dt E_8, \dt E_8(a_1),\dt E_8(a_2),\dt E_7+\dt A_1\}$. In particular, we know that $O_u\cap\mf B_0^F\dot w_{49}\mf B_0^F=\varnothing$ for $u\in\dt E_7^F$, so we have
\[m(u,w_{49})=0\quad\text{for any }u\in\dt E_7^F.\]
Now let us consider the right hand side of \eqref{GCarpath}, for $u$ and $w_{49}$ as above. To do this, we thus need a detailed information on the values $\chi_x(u)$ for $x\in X(\mf W)$ and $u\in\dt E_7^F$.
\end{0}

\begin{Lm}\label{NonVanishRxNG}
Let $x\in X(\mf W)$, and assume that $\chi_x(u)\neq0$ for some $u\in\mf G^F_{\mathrm{uni}}$. Then there exists some $\ii\in\cN_\mf G$ such that $A_x$ is isomorphic to $A_\ii$.
\end{Lm}

\begin{proof}
By the description in \Cref{ParUCSHC}, $A_x$ is parametrised by some triple $(J,\phi,\zeta)\in\mathfrak S_\mf W$, that is, $A_x$ is (isomorphic to) a simple direct summand of $\ind_{\mf L_J}^{\mf G}(A_0)$ where $A_0$ corresponds to $\zeta\in\mathfrak S_{\mf W_J}^\circ$. Since $A_0\in{\hat{\mf L}_J}^\circ$, there exists some cuspidal pair $(\Sigma,\cE)$ for $\mf L_J$ such that
\[A_0\cong{\IC(\overline\Sigma,\cE)[\dim\Sigma]}^{\#\mf L_J}\]
(see \cite[Proposition 3.12]{LuCS1}, \cite[\S 2]{LuIC}). In view of the cleanness of cuspidal character sheaves (\cite[Theorem 23.1]{LuCS5}, \cite{LuclCS}), we know that $A_0$ is $0$ outside of $\Sigma$. Now we have
\[\supp A_x=\bigcup_{g\in\mf G}g(\supp A_0)R_\mathrm u(\mf P_J)g^{-1},\]
see, e.g., \cite[2.9]{Luvaluni}. Thus, our assumption that $\chi_x(u)\neq0$ for some $u\in\mf G^F_{\mathrm{uni}}$ implies that $\Sigma\cap{{(\mf L_J)}_{\mathrm{uni}}}\neq\varnothing$. By the definition of cuspidal pairs \cite[\S 2]{LuIC}, $\Sigma$ is of the form ${\mf Z(\mf L_J)}^\circ.\cC$ for some conjugacy class $\cC\subseteq\mf L_J$, so we can assume that $\cC=\cO_0$ is a \emph{unipotent} conjugacy class in $\mf L_J$ and that $\cE$ is the inverse image of some $\mf L_J$-equivariant irreducible local system $\cE_0$ on $\cO_0$ under the canonical map ${\mf Z(\mf L_J)}^\circ.\cO_0\rightarrow\cO_0$, that is, we have $\jj:=(\mf L_J,\cO_0,\cE_0)\in\cM_\mf G$. The associated perverse sheaf $K_\jj$ (see \Cref{LuIC}) is isomorphic to $\ind_{\mf L_J}^{\mf G}(A_0)$ by \cite[Prop. 4.5]{LuIC}, so $A_x$ must be isomorphic to $A_\ii$ for some $\ii\in\tau^{-1}(\jj)\subseteq\cN_\mf G$.
\end{proof}

\begin{Lm}\label{VanishRx}
Let $\ii=(\cO,\cE)\in\cN_\mf G=\cN_\mf G^F$ be such that $\cO\notin\{\dt E_8,\dt E_8(a_1),\dt E_8(a_2),\dt E_7+\dt A_1,\dt E_7\}$. Then, for any $u\in\dt E_7^F$, we have
\[X_\ii(u)=0.\]
\end{Lm}

\begin{proof}
Let us assume that $X_\ii(u)\neq0$. By \Cref{LusztigAlgorithm} (and the definition of the $Y_{\ii'}$ for $\ii'\in\cN_\mf G^F$), there must exist some $\ii'=(\cO',\cE')\in\cN_\mf G^F$ with $\cO'=\dt E_7$ and with $p_{\ii',\ii}\neq0$ in \Cref{LusztigAlgorithm}(b). It then follows from \Cref{CorLusztigAlgorithm} that we must have $\dt E_7\subseteq\overline{\cO}$, which happens precisely for $\cO\in\{\dt E_8,\dt E_8(a_1),\dt E_8(a_2),\dt E_7+\dt A_1,\dt E_7\}$.
\end{proof}

\begin{0}\label{GreenFunctions}
We can now proceed with the evaluation of the right hand side of the equation \eqref{GCarpath} in \Cref{Hecke}, for any $g=u\in\dt E_7^F$ and with $w=w_{49}$ as in \eqref{Defw49}. In view of \Cref{NonVanishRxNG} and \Cref{VanishRx}, we are reduced to considering those $x\in X(\mf W)$ for which $A_x\cong A_\ii$ where $\ii=(\cO,\cE)\in\cN_\mf G=\cN_\mf G^F$ is such that $\cO\in\{\dt E_8,\dt E_8(a_1),\dt E_8(a_2),\dt E_7+\dt A_1,\dt E_7\}$. For such an $\ii$, let $\jj=\tau(\ii)=(\mf L_J,\cO_0,\cE_0)\in\cM_\mf G^F$, with $J\subseteq\Pi$. Recall from \Cref{NGE8p3} that there are three possibilities for $J$, namely, $J=\varnothing$, $J=\{\alpha_1,\alpha_2,\ldots,\alpha_6\}$ or $J=\Pi$. In the case where $J=\Pi$, we have $A_\ii\cong K_\jj$, and this is a cuspidal character sheaf on $\mf G$. We necessarily have $\cO=\cO_0=\dt E_7+\dt A_1$, see \cite[p.~337]{Sp}. By the cleanness of cuspidal character sheaves (\cite[23.1]{LuCS5}, \cite{LuclCS}), we already know that $X_\ii(u)=0$ (thus also $\chi_x(u)=0$) for $u\in\dt E_7^F$. So it remains to consider the following two cases (labelled as in \Cref{NGE8p3}):
\begin{itemize}
\item[(a)] $J=\varnothing$, that is, $\ii$ is in the image of the ordinary Springer correspondence \eqref{OrdSpring}, so it is parametrised by some explicitly known $\phi\in\Irr(\mf W)$; see \Cref{OrdSpringE8p3} below.
\item[(c)] $J=\{\alpha_1,\alpha_2,\ldots,\alpha_6\}$; see \Cref{LuConj} below.
\end{itemize}
A detailed description of the generalised Springer correspondence with respect to (a) and (c) is given in \Cref{TblGenSpring}, with the following conventions: Let us consider a fixed line, and let $\cO\subseteq\mf G$ be the unipotent class appearing in the first column in that line. We assume to have fixed a representative $u\in\cO$, and then $d_u$ is the dimension of the variety of all Borel subgroups of $\mf G$ which contain $u$. As for the notation $a:b$ in the last two columns, $a$ denotes the irreducible character of $A_\mf G(u)$ describing the local system on $\cO$, and $b$ is the irreducible character of the relative Weyl group $\mf W^{\Pi/J}$ such that $b$ is mapped to $(\cO,a)$ under the generalised Springer correspondence. In the last column, we do not yet know which of $\rho$, $\rho'$ belongs to which place, only that $\omega:\rho$ and $\omega^2:\rho$ both appear exactly once (and analogously for $\rho'$). We also provide the position of the corresponding unipotent character in {\sffamily CHEVIE} \cite{MiChv} via \eqref{ReformGenSpringE8p3} in \Cref{AiAxE8p3}, as it appears in \texttt{UnipotentCharacters($\mf W$)}, with the convention $\omega=\dt E3$.
{\renewcommand{\arraystretch}{1.4}
\begin{table}[htbp]
\centering
\begin{tabular}{c|c|c|c|c}
$u$ & $d_u$ & $A_\mf G(u)$ & $J=\varnothing$, $\mf W^{\Pi/J}=\mf W$ & $J=\dt E_6$, $\mf W^{\Pi/J}\cong W(\dt G_2)$ \\
\hline
\multirow{2}{*}{$\dt E_8$} & \multirow{2}{*}{$0$} & \multirow{2}{*}{$\modulo{\Z}{3\Z}$} & \multirow{2}{*}{$1:1_x$ (\#1)} & $\omega:1$ (\#142) \\
& & & & $\omega^2:1$ (\#148) \\
\hline
\multirow{2}{*}{$\dt E_8(a_1)$} & \multirow{2}{*}{$1$} & \multirow{2}{*}{$\modulo{\Z}{3\Z}$} & \multirow{2}{*}{$1:8_z$ (\#68)} & $\omega:\epsilon$ (\#143) \\
& & & & $\omega^2:\epsilon$ (\#149) \\
\hline
$\dt E_8(a_2)$ & $2$ & $\{1\}$ & $1:35_x$ (\#5) & --- \\
\hline
\multirow{2}{*}{$\dt E_7+\dt A_1$} & \multirow{2}{*}{$3$} & \multirow{2}{*}{$\modulo{\Z}{6\Z}$} & $1:112_z$ (\#72) & $\omega:\rho/\rho'$ (\#152/\#153) \\
& & & $-1:28_x$ (\#3) & $\omega^2:\rho/\rho'$ (\#146/\#147) \\
\hline
\multirow{2}{*}{$\dt E_7$} & \multirow{2}{*}{$4$} & \multirow{2}{*}{$\modulo{\Z}{3\Z}$} & \multirow{2}{*}{$1:84_x$ (\#10)} & $\omega:\rho'/\rho$ (\#153/\#152) \\
& & & & $\omega^2:\rho'/\rho$ (\#147/\#146)
\end{tabular}
\caption{Parts of the generalised Springer correspondence for $\dt E_8$, $p=3$, due to \cite{Sp}. The number in parentheses refers to the {\sffamily CHEVIE} number of the corresponding unipotent character via \eqref{ReformGenSpringE8p3} in \Cref{AiAxE8p3}, with the convention $\omega=\dt E3$.}
\label{TblGenSpring}
\end{table}
}
\end{0}

\begin{0}\label{OrdSpringE8p3}
Let us assume that $\ii=(\cO,\cE)\in\cN_\mf G^F$ is as in \Cref{GreenFunctions}(a), so that $\jj=\tau(\ii)=(\mf T_0,\{1\},1)$. Let $x\in X(\mf W)$ be such that $A_x\cong A_\ii$. By \Cref{CorLusztigAlgorithm}(iii), we know that whenever $p_{\ii',\ii}\neq0$ for some $\ii'\in\cN_\mf G^F$, we must have $\tau(\ii')=\jj$. The only such $\ii'=(\cO',\cE')$ with $\cO'=\dt E_7$ is $\ii'=(\dt E_7,1)$. So for $u\in\dt E_7^F$, in order to express $X_\ii(u)$ in terms of the $Y_{\ii'}(u)$, we only need to know $p_{\ii',\ii}$ with $\ii'=(\dt E_7,1)$. This is easily obtained using the {\sffamily CHEVIE} \cite{MiChv} functions \texttt{UnipotentClasses} and \texttt{ICCTable}. (For any $\ii$ as above, we have $p_{\ii',\ii}\in\{0,1\}$.) In combination with \Cref{ChixXi}, we can give the coefficient of $\chi_x$ at $Y_{\ii'}$; see \Cref{TblGreen}.
{\renewcommand{\arraystretch}{1.4}
\begin{table}[htbp]
\centering
\begin{tabular}{c|c|c|c}
\hline
\multicolumn{3}{c|}{$\phi\in\Irr(\mf W)$ so that $x=x_\phi$ \eqref{IrrWXW}}  & Coefficient of \\
Lusztig \cite{Luchars} & Carter \cite{C} & {\sffamily CHEVIE} \cite{MiChv} & $\chi_x$ at $Y_{\ii'}$ \\
\hline
$1_x$ & $\phi_{1,0}$ & \#1 & $1$ \\
\hline
$8_z$ & $\phi_{8,1}$ & \#68 & $q$ \\
\hline
$35_x$ & $\phi_{35,2}$ & \#5 & $q^2$ \\
\hline
$28_x$ & $\phi_{28,8}$ & \#3 & $0$ \\
\hline
$112_z$ & $\phi_{112,3}$ & \#72 & $q^3$ \\
\hline
$84_x$ & $\phi_{84,4}$ & \#10 & $q^4$ \\
\end{tabular}
\caption{Coefficients of $\chi_x$ at $Y_{\ii'}$ for the relevant $x=x_\phi\in X(\mf W)$, where $\ii'=(\dt E_7,1)$. The {\sffamily CHEVIE} column refers to the position of the corresponding unipotent principal series character in \texttt{UnipotentCharacters($\mf W$)}.}
\label{TblGreen}
\end{table}
}
\end{0}

\begin{0}\label{LuConj}
Now let us consider the case (c) in \Cref{GreenFunctions}. So let $J=\{\alpha_1,\alpha_2,\ldots,\alpha_6\}\subseteq\Pi$, and let $\jj=(\mf L_J,\cO_0,\cE_0)$ be one of the two associated elements of $\cM_\mf G^F$. Thus, $\modulo{\mf L_J}{{\mf Z(\mf L_J)}^\circ}$ is the simple (adjoint) group of type $\dt E_6$, and $\scW_\jj=W_\mf G(\mf L_J)\cong\mf W^{\Pi/J}\cong W(\dt G_2)$ is the dihedral group of order $12$. We use the notation of \Cref{WG2} for the irreducible characters of $\mf W^{\Pi/J}$.
\begin{assumption}
From now on we fix a primitive $3$rd root of unity $\omega\in\Qlbarunits$.
\end{assumption}
Recall from \Cref{IntroLuConj} that the above two elements $\jj$ of $\cM_\mf G$ are $(\mf L_J,\cO_0,\omega)$ and $(\mf L_J,\cO_0,\omega^2)$ (with the conventions in \Cref{AGu}), where $\cO_0\subseteq\mf L_J$ is the regular unipotent class. The block of $\cN_\mf G^F$ corresponding to $\jj=(\mf L_J,\cO_0,\omega)$ is given by
\[\tau^{-1}(\jj)=\{(\dt E_8,\omega),(\dt E_8(a_1),\omega),(\dt E_7+\dt A_1,\omega),(\dt E_7,\omega),(\dt E_6+\dt A_1,\omega),(\dt E_6,\omega)\}.\]
In view of \Cref{TblSpaltE8E6}, the generalised Springer correspondence with respect to this block is
\begin{align*}
1\leftrightarrow(\dt E_8,\omega),\;\epsilon\leftrightarrow(\dt E_8(a_1),\omega),\;&\epsilon'\leftrightarrow(\dt E_6+\dt A_1,\omega),\;\sgn\leftrightarrow(\dt E_6,\omega), \\
\theta'\leftrightarrow(\dt E_7+\dt A_1,\omega&),\;\theta''\leftrightarrow(\dt E_7,\omega),
\end{align*}
where $\{\theta',\theta''\}=\{\rho,\rho'\}$. So we get the following possibilities with respect to $\jj=(\mf L_J,\cO_0,\omega)$:
\begin{enumerate}
\item[($\omega$a)] $\rho\leftrightarrow(\dt E_7+\dt A_1,\omega)$ and $\rho'\leftrightarrow(\dt E_7,\omega)$, or
\item[($\omega$b)] $\rho\leftrightarrow(\dt E_7,\omega)$ and $\rho'\leftrightarrow(\dt E_7+\dt A_1,\omega)$.
\end{enumerate}
Analogously, replacing $\omega$ by $\omega^2$ in the above discussion, one of the following holds with respect to $\jj=(\mf L_J,\cO_0,\omega^2)$:
\begin{enumerate}
\item[($\omega^2$a)] $\rho\leftrightarrow(\dt E_7+\dt A_1,\omega^2)$ and $\rho'\leftrightarrow(\dt E_7,\omega^2)$, or
\item[($\omega^2$b)] $\rho\leftrightarrow(\dt E_7,\omega^2)$ and $\rho'\leftrightarrow(\dt E_7+\dt A_1,\omega^2)$.
\end{enumerate}
In order to prove Lusztig's conjecture \eqref{OpenCaseE8p3LuConj} (see \Cref{IntroLuConj}), we thus have to show that ($\omega$a) and ($\omega^2$a) hold. We do this by excluding ($\omega$b) and ($\omega^2$b).
\end{0}

\begin{0}\label{MatricesG2}
We keep the setting and notation of \Cref{LuConj}. Thus, we have $\jj=(\mf L_J,\cO_0,\zeta_3)\in\cM_\mf G^F$ with $J=\{\alpha_1,\alpha_2,\ldots,\alpha_6\}$ and $\zeta_3\in\{\omega,\omega^2\}$. For $\ii\in\tau^{-1}(\jj)$, we want to evaluate $X_\ii$ at $u\in\dt E_7^F$, so we need to know the coefficient $p_{\ii',\ii}$ in \Cref{LusztigAlgorithm}(b) where $\ii'=(\dt E_7,\zeta_3)\in\tau^{-1}(\jj)$. But note that $p_{\ii',\ii}$ depends on whether $(\zeta_3$a) or $(\zeta_3$b) in \Cref{LuConj} holds. Since the corresponding coefficient matrices $P_{(\zeta_3\mathrm{a})}$, $P_{(\zeta_3\mathrm{b})}$ (indexed according to whether ($\zeta_3$a) or ($\zeta_3$b) holds) only depend on $\scW_\jj\cong W(\dt G_2)$ (and not on $\mf G$ or $\mf L_J$), we can just take them from \cite[5.2]{LugenSpringer}. They are given by
\[P_{(\zeta_3\mathrm{a})}=\begin{pmatrix}
1 & 0 & 0 & 0 & 0 & 0 \\
0 & 1 & 0 & 0 & 0 & 0 \\
q^2 & 0 & 1 & 0 & 0 & 0 \\
q^2 & q^2 & 1 & 1 & 0 & 0 \\
q^6 & 0 & q^4 & q^4 & 1 & 0 \\
q^6 & q^8 & q^4+q^8 & q^4+q^6 & 1 & 1
\end{pmatrix}
\quad\longleftrightarrow\quad(1,\epsilon,\rho,\rho',\epsilon',\sgn)
\]
and
\[P_{(\zeta_3\mathrm{b})}=\begin{pmatrix}
1 & 0 & 0 & 0 & 0 & 0 \\
0 & 1 & 0 & 0 & 0 & 0 \\
0 & q & 1 & 0 & 0 & 0 \\
q^3 & q & 1 & 1 & 0 & 0 \\
q^6 & 0 & q^5 & q^3 & 1 & 0 \\
q^6 & q^8 & q^5+q^7 & q^3+q^7 & 1 & 1
\end{pmatrix}
\quad\longleftrightarrow\quad(1,\epsilon,\rho',\rho,\epsilon',\sgn),
\]
where the vector on the right describes the ordering of the irreducible characters of $W(\dt G_2)$ corresponding to the rows and columns of $P_{(\zeta_3\mathrm{a})}$, $P_{(\zeta_3\mathrm{b})}$, respectively. Hence, for $\ii'=(\dt E_7,\zeta_3)\in\tau^{-1}(\jj)$ as above, the coefficients of the various $X_\ii$ at $Y_{\ii'}$ are the entries in the $4$th line of either of the matrices $P_{(\zeta_3\mathrm{a})}$, $P_{(\zeta_3\mathrm{b})}$. Recall (\ref{AxAi}, \ref{AiAxE8p3}) that the $x\in X(\mf W)$ corresponding to a given $\ii\in\{(\dt E_7+\dt A_1,\zeta_3),(\dt E_7,\zeta_3)\}$ depends on whether ($\zeta_3$a) or ($\zeta_3$b) holds. Using \Cref{ChixXi} and the matrices $P_{(\zeta_3\mathrm{a})}$, $P_{(\zeta_3\mathrm{b})}$, we find the coefficient of $Y_{\ii'}$ in the relevant $\chi_x$. These are given in \Cref{TblValuesRx}, with the following conventions: In the first column, we describe $x\in X(\mf W)$ by giving the corresponding element of $\mathfrak S_\mf W$ via \Cref{Parcompatible}; here, we write $\dt E_6$ instead of $J=\{\alpha_1,\alpha_2,\ldots,\alpha_6\}$. The second column contains the {\sffamily CHEVIE} \cite{MiChv} number of the unipotent character $\rho_x$. In the third column, the index of $\scF$ specifies that family by being the special character in it \cite[(4.1.4)]{Luchars}, and we also give the label of $\rho_x$ in this family. The $\zeta_3$ in the last two columns is always meant to be the unique element in $\{\omega,\omega^2\}$ which appears in the line considered.
{\renewcommand{\arraystretch}{1.4}
\begin{table}[htbp]
\centering
\begin{tabular}{c|c|c|c|c}
$x\in X(\mf W)$ & \multirow{2}{*}{{\sffamily CHEVIE}} & \multirow{2}{*}{Family:Label} & \multicolumn{2}{c}{Coefficient of $\chi_x$ at $Y_{\ii'}$} \\
via $X(\mf W)\cong\mathfrak S_\mf W$ & & & Case $(\zeta_3$a) & Case $(\zeta_3$b) \\
\hline
$(\dt E_6,1,\omega)$ & \#142 & $\scF_{1400_z}:(g_3,\omega)$ & $q^5$ & $q^6$ \\
\hline
$(\dt E_6,1,\omega^2)$ & \#148 & $\scF_{1400_z}:(g_3,\omega^2)$ & $q^5$ & $q^6$ \\
\hline
$(\dt E_6,\epsilon,\omega)$ & \#143 & $\scF_{1400_x}:(g_3,\omega)$ & $q^6$ & $q^5$ \\
\hline
$(\dt E_6,\epsilon,\omega^2)$ & \#149 & $\scF_{1400_x}:(g_3,\omega^2)$ & $q^6$ & $q^5$ \\
\hline
$(\dt E_6,\rho,\omega)$ & \#152 & $\scF_{4480_y}:(g_3,\omega)$ & $q^6$ & $q^7$ \\
\hline
$(\dt E_6,\rho,\omega^2)$ & \#146 & $\scF_{4480_y}:(g_3,\omega^2)$ & $q^6$ & $q^7$ \\
\hline
$(\dt E_6,\rho',\omega)$ & \#153 & $\scF_{4480_y}:(g_6,\omega)$ & $q^7$ & $q^6$ \\
\hline
$(\dt E_6,\rho',\omega^2)$ & \#147 & $\scF_{4480_y}:(g_6,\omega^2)$ & $q^7$ & $q^6$ \\
\end{tabular}
\caption{Coefficients of $\chi_x$ at $Y_{\ii'}$ for the relevant $x\in X(\mf W)$, where $\ii'=(\dt E_7,\zeta_3)$ with $\zeta_3\in\{\omega,\omega^2\}$ the same as in the first and third column.}
\label{TblValuesRx}
\end{table}
}
\end{0}

\begin{Thm}\label{LuConjTrue}
In the setting of \Cref{LuConj}, ($\omega$\emph a) and ($\omega^2$\emph a) hold, that is, Lusztig's conjecture \eqref{OpenCaseE8p3LuConj} (see \Cref{IntroLuConj}) is true.
\end{Thm}

\begin{proof}
For $x\in X(\mf W)$ and $w\in\mf W$, let us set
\[c_x(w):=\sum_{\phi\in\Irr(\mf W)}\{x,x_\phi\}\Trace(T_w,V_\phi).\]
Hence, with $w_{49}\in\mf W$ as defined in \eqref{Defw49}, \eqref{GCarpath} gives
\begin{equation}\label{muw49Rx}
0=m(u,w_{49})=\sum_{x\in X(\mf W)}c_x(w_{49})\nu_x\chi_x(u)\quad\text{for any }u\in\dt E_7^F.
\end{equation}
In view of \Cref{VanishRx} and \Cref{TblGenSpring}, we only need to look at the values $c_x(w_{49})$ for $x\in X(\mf W)$ as in \Cref{TblGreen} and \Cref{TblValuesRx}. The $c_x(w_{49})$ for those $x$ are given in \Cref{Tblcxw}.
{\renewcommand{\arraystretch}{1.4}
\begin{table}[htbp]
\centering
\begin{tabular}{c|c|c}
$x\in X(\mf W)$ via $X(\mf W)\cong\mathfrak S_\mf W$ & {\sffamily CHEVIE} \cite{MiChv}  & $c_x(w_{49})$ \\
\hline
$1_x$ & \#1 & $q^{14}$ \\
\hline
$8_z$ & \#68 & $q^{12}$ \\
\hline
$35_x$ & \#5 & $-q^{11}$ \\
\hline
$28_x$ & \#3 & $0$ \\
\hline
$112_z$ & \#72 & $0$ \\
\hline
$84_x$ & \#10 & $-q^{10}$ \\
\hline
$(\dt E_6,1,\omega)$ & \#142 & $q^9$ \\
\hline
$(\dt E_6,1,\omega^2)$ & \#148 & $q^9$ \\
\hline
$(\dt E_6,\epsilon,\omega)$ & \#143 & $-q^8$ \\
\hline
$(\dt E_6,\epsilon,\omega^2)$ & \#149 & $-q^8$ \\
\hline
$(\dt E_6,\rho,\omega)$ & \#152 & $0$ \\
\hline
$(\dt E_6,\rho,\omega^2)$ & \#146 & $0$ \\
\hline
$(\dt E_6,\rho',\omega)$ & \#153 & $0$ \\
\hline
$(\dt E_6,\rho',\omega^2)$ & \#147 & $0$
\end{tabular}
\caption{$c_x(w_{49})$ for the relevant $x\in X(\mf W)$}
\label{Tblcxw}
\end{table}
}

Now let us consider the elements $\ii'=(\cO',\cE')\in\cN_\mf G^F$ with $\cO'=\dt E_7$. For $u_0\in\dt E_7$, we have $A_\mf G(u_0)=\langle\overline u_0\rangle\cong\modulo{\Z}{3\Z}$. Thus, with our conventions in \Cref{AGu}, there are three $\ii'$ as above: $(\dt E_7,1)$, $(\dt E_7,\omega)$ and $(\dt E_7,\omega^2)$. For $\varsigma\in\Irr(A_\mf G(u_0))$, let us set
\[\gamma_\varsigma:=Y_{(\dt E_7,\varsigma)}|_{\dt E_7^F}\colon\dt E_7^F\rightarrow\Qlbar.\]
By \Cref{LusztigAlgorithm} (and since, by definition, the $Y_{(\dt E_7,\varsigma)}$ vanish outside of $\dt E_7^F$), $\gamma_1$, $\gamma_\omega$ and $\gamma_{\omega^2}$ are linearly independent functions $\dt E_7^F\rightarrow\Qlbar$. Using Tables \ref{TblGreen}, \ref{TblValuesRx}, \ref{Tblcxw} to evaluate \eqref{muw49Rx} for the various $u\in\dt E_7^F$ (depending on which two of ($\omega$a), ($\omega$b), ($\omega^2$a), ($\omega^2$b) hold), we obtain an expression of the zero function $\dt E_7^F\rightarrow\Qlbar$ as a linear combination of $\gamma_1$, $\gamma_\omega$ and $\gamma_{\omega^2}$, as their coefficients are independent of the chosen $u\in\dt E_7^F$. Due to the linear independence of $\gamma_1$, $\gamma_\omega$, $\gamma_{\omega^2}$, we know that any of their coefficients must be $0$. If $n\in\{1,2,\ldots,166\}$ is the {\sffamily CHEVIE} number of $x\in X(\mf W)$ as in \Cref{Tblcxw}, it will be convenient to write $\chi_n$, $\nu_n$ instead of $\chi_x$, $\nu_x$, respectively. Now let us assume, if possible, that ($\omega$b) holds. Then the coefficient of $\gamma_\omega$ in \eqref{muw49Rx} is $q^{15}\nu_{142}-q^{13}\nu_{143}$, and this must be $0$. However, $\nu_{142}$ and $\nu_{143}$ are roots of unity and $q\geqslant3$, a contradiction. So ($\omega$a) must hold. Similarly, if we assume that ($\omega^2$b) holds, the coefficient of $\gamma_{\omega^2}$ in \eqref{muw49Rx} is $q^{15}\nu_{148}-q^{13}\nu_{149}=0$ which contradicts the fact that both $\nu_{148}$ and $\nu_{149}$ are roots of unity. So ($\omega^2$a) holds as well, as claimed.
\end{proof}

\begin{Rem}
(a) Instead of evaluating the equation \eqref{GCarpath} in \Cref{Hecke} at elements of $\dt E_7^F$, we could have also considered the class $\dt E_7+\dt A_1\subseteq\mf G$ and, e.g., taken $w$ as a Coxeter element of $\mf W$. This leads to a similar contradiction when Lusztig's conjecture is assumed to be false.

(b) In contrast, if we evaluate \eqref{GCarpath} with the correct assumptions ($\omega$a) and ($\omega^2$a) in \Cref{LuConj}, for $u\in\mf G^F_{\mathrm{uni}}$ and $w\in\mf W$ such that $m(u,w)=0$, we see that the coefficient of any fixed $\gamma_\varsigma q^e$ ($\varsigma\in\Irr(A_\mf G(u_0))$, $e\geqslant0$) on the right hand side of \eqref{GCarpath} will never consist of exactly one root of unity $\nu_x$, $x\in X(\mf W)$, so we will (of course) not run into a contradiction of the type as in the proof of \Cref{LuConjTrue}.
\end{Rem}

\begin{Ack}
I deeply thank my PhD advisor Meinolf Geck for suggesting to me to work on this problem, as well as for a very careful reading of the manuscript, including numerous discussions and detailed recommendations on how to improve the paper. I am especially grateful to George Lusztig for comments on the paper and for pointing out an additional reference. This work was supported by the Deutsche Forschungsgemeinschaft (DFG, German Research Foundation) --– Project-ID 286237555 – TRR 195.
\end{Ack}

\bibliographystyle{abbrv}

\bibliography{E8p3}

\begin{thebibliography}{10}

\bibitem{BBD}
A.~A. Beilinson, J.~Bernstein, and P.~Deligne.
\newblock Faisceaux pervers.
\newblock In {\em Analysis and topology on singular spaces, {I} ({L}uminy,
  1981)}, volume 100 of {\em Ast\'{e}risque}, pages 5--171. Soc. Math. France,
  Paris, 1982.

\bibitem{C}
R.~W. Carter.
\newblock {\em Finite groups of {L}ie type}.
\newblock Pure and Applied Mathematics (New York). John Wiley \& Sons, Inc.,
  New York, 1985.
\newblock A Wiley-Interscience Publication.

\bibitem{CR1}
C.~W. Curtis and I.~Reiner.
\newblock {\em Methods of representation theory. {V}ol. {I}}.
\newblock Pure and Applied Mathematics. John Wiley \& Sons, Inc., New York,
  1981.

\bibitem{DeligneWeyl2}
P.~Deligne.
\newblock La conjecture de {W}eil. {II}.
\newblock {\em Inst. Hautes \'{E}tudes Sci. Publ. Math.}, 52:137--252, 1980.

\bibitem{DL}
P.~Deligne and G.~Lusztig.
\newblock Representations of reductive groups over finite fields.
\newblock {\em Ann. of Math. (2)}, 103(1):103--161, 1976.

\bibitem{DMParam}
F.~Digne and J.~Michel.
\newblock On {L}usztig's parametrization of characters of finite groups of
  {L}ie type.
\newblock {\em Ast\'{e}risque}, 181-182:113--156, 1990.

\bibitem{DM20}
F.~Digne and J.~Michel.
\newblock {\em Representations of finite groups of {L}ie type}, volume~95 of
  {\em London Mathematical Society Student Texts}.
\newblock Cambridge University Press, Cambridge, second edition, 2020.

\bibitem{Eft}
M.~Eftekhari.
\newblock Descente de {S}hintani des faisceaux caract\`eres.
\newblock {\em C. R. Acad. Sci. Paris S\'{e}r. I Math.}, 318(4):305--308, 1994.

\bibitem{GCH}
M.~Geck.
\newblock Some applications of {CHEVIE} to the theory of algebraic groups.
\newblock {\em Carpathian J. Math.}, 27(1):64--94, 2011.

\bibitem{GM}
M.~Geck and G.~Malle.
\newblock {\em The character theory of finite groups of {L}ie type}, volume 187
  of {\em Cambridge Studies in Advanced Mathematics}.
\newblock Cambridge University Press, Cambridge, 2020.

\bibitem{GeMi}
M.~Geck and J.~Michel.
\newblock `{G}ood' elements of finite {C}oxeter groups and representations of
  {I}wahori-{H}ecke algebras.
\newblock {\em Proc. London Math. Soc. (3)}, 74(2):275--305, 1997.

\bibitem{GMP}
M.~Goresky and R.~MacPherson.
\newblock Intersection homology. {II}.
\newblock {\em Invent. Math.}, 72(1):77--129, 1983.

\bibitem{LuCoxFrob}
G.~Lusztig.
\newblock Coxeter orbits and eigenspaces of {F}robenius.
\newblock {\em Invent. Math.}, 38(2):101--159, 1976.

\bibitem{Lurepchev}
G.~Lusztig.
\newblock {\em Representations of finite {C}hevalley groups}, volume~39 of {\em
  CBMS Regional Conference Series in Mathematics}.
\newblock American Mathematical Society, Providence, R.I., 1978.
\newblock Expository lectures from the CBMS Regional Conference held at
  Madison, Wis., August 8--12, 1977.

\bibitem{LuGreenPol}
G.~Lusztig.
\newblock Green polynomials and singularities of unipotent classes.
\newblock {\em Adv. in Math.}, 42(2):169--178, 1981.

\bibitem{Luchars}
G.~Lusztig.
\newblock {\em Characters of reductive groups over a finite field}, volume 107
  of {\em Annals of Mathematics Studies}.
\newblock Princeton University Press, Princeton, NJ, 1984.

\bibitem{LuIC}
G.~Lusztig.
\newblock Intersection cohomology complexes on a reductive group.
\newblock {\em Invent. Math.}, 75(2):205--272, 1984.

\bibitem{LuCS1}
G.~Lusztig.
\newblock Character sheaves. {\rom{1}}.
\newblock {\em Adv. in Math.}, 56:193--237, 1985.

\bibitem{LuCS2}
G.~Lusztig.
\newblock Character sheaves. {\rom{2}}.
\newblock {\em Adv. in Math.}, 57:226--265, 1985.

\bibitem{LuCS3}
G.~Lusztig.
\newblock Character sheaves. {\rom{3}}.
\newblock {\em Adv. in Math.}, 57:266--315, 1985.

\bibitem{LuCS4}
G.~Lusztig.
\newblock Character sheaves. {\rom{4}}.
\newblock {\em Adv. in Math.}, 59:1--63, 1986.

\bibitem{LuCS5}
G.~Lusztig.
\newblock Character sheaves. {\rom{5}}.
\newblock {\em Adv. in Math.}, 61:103--155, 1986.

\bibitem{Luvaluni}
G.~Lusztig.
\newblock On the character values of finite {C}hevalley groups at unipotent
  elements.
\newblock {\em J. Algebra}, 104(1):146--194, 1986.

\bibitem{LuIntroCS}
G.~Lusztig.
\newblock Introduction to character sheaves.
\newblock In {\em The {A}rcata {C}onference on {R}epresentations of {F}inite
  {G}roups ({A}rcata, {C}alif., 1986)}, volume~47 of {\em Proc. Sympos. Pure
  Math.}, pages 165--179. Amer. Math. Soc., Providence, RI, 1987.

\bibitem{LuCSDC4}
G.~Lusztig.
\newblock Character sheaves on disconnected groups. {\rom{4}}.
\newblock {\em Represent. Theory}, 8:145--178, 2004.

\bibitem{LuWeylUni}
G.~Lusztig.
\newblock From conjugacy classes in the {W}eyl group to unipotent classes.
\newblock {\em Represent. Theory}, 15:494--530, 2011.

\bibitem{LuEllip}
G.~Lusztig.
\newblock Elliptic elements in a {W}eyl group: a homogeneity property.
\newblock {\em Represent. Theory}, 16:127--151, 2012.

\bibitem{LuclCS}
G.~Lusztig.
\newblock On the cleanness of cuspidal character sheaves.
\newblock {\em Mosc. Math. J.}, 12(3):621--631, 669, 2012.

\bibitem{LuRCS}
G.~Lusztig.
\newblock Restriction of a character sheaf to conjugacy classes.
\newblock {\em Bull. Math. Soc. Sci. Math. Roumanie (N.S.)},
  58(106)(3):297--309, 2015.

\bibitem{LuUchCatCen}
G.~Lusztig.
\newblock Unipotent representations as a categorical centre.
\newblock {\em Represent. Theory}, 19:211--235, 2015.

\bibitem{LugenSpringer}
G.~Lusztig.
\newblock On the generalized {S}pringer correspondence.
\newblock In {\em Representations of reductive groups}, volume 101 of {\em
  Proc. Sympos. Pure Math.}, pages 219--253. Amer. Math. Soc., Providence, RI,
  2019.

\bibitem{LuSp}
G.~Lusztig and N.~Spaltenstein.
\newblock On the generalized {S}pringer correspondence for classical groups.
\newblock In {\em Algebraic groups and related topics ({K}yoto/{N}agoya,
  1983)}, volume~6 of {\em Adv. Stud. Pure Math.}, pages 289--316.
  North-Holland, Amsterdam, 1985.

\bibitem{MiChv}
J.~Michel.
\newblock The development version of the {CHEVIE} package of {{GAP}3}.
\newblock {\em J. Algebra}, 435:308--336, 2015.

\bibitem{ShGeomOrb}
T.~Shoji.
\newblock Geometry of orbits and {S}pringer correspondence.
\newblock In {\em Orbites unipotentes et repr\'{e}sentations, I}, volume 168 of
  {\em Ast\'{e}risque}, pages 61--140. Soci\'et\'e math\'ematique de France,
  1988.

\bibitem{Sh1}
T.~Shoji.
\newblock Character sheaves and almost characters of reductive groups.
\newblock {\em Adv. Math.}, 111:244--313, 1995.

\bibitem{Sh2}
T.~Shoji.
\newblock Character sheaves and almost characters of reductive groups. {II}.
\newblock {\em Adv. Math.}, 111:314--354, 1995.

\bibitem{Sp}
N.~Spaltenstein.
\newblock On the generalized {S}pringer correspondence for exceptional groups.
\newblock In {\em Algebraic groups and related topics ({K}yoto/{N}agoya,
  1983)}, volume~6 of {\em Adv. Stud. Pure Math.}, pages 317--338.
  North-Holland, Amsterdam, 1985.

\bibitem{SCorr}
T.~A. Springer.
\newblock Trigonometric sums, {G}reen functions of finite groups and
  representations of {W}eyl groups.
\newblock {\em Invent. Math.}, 36:173--207, 1976.

\end{thebibliography}

\end{document}